\documentclass[11pt]{amsart}
\usepackage{amsmath,amsfonts,amssymb,amscd,mathtools,bm}
\usepackage[alphabetic]{amsrefs}
\usepackage[colorlinks=true, linkcolor=blue,
  citecolor=blue, urlcolor=blue]{hyperref}
\usepackage{tikz}
\usepackage{tikz-cd} 
\usepackage{enumitem}
\usepackage{multirow}
\usepackage{caption}
\usepackage{subcaption}

\usepackage{hyperref}
\usepackage[capitalize,nameinlink]{cleveref}

\def\newthm#1#2{\newtheorem{#1}[dummy]{#2}%
  \expandafter\def\csname#2\endcsname##1{\hyperref[#1:##1]{#2~\ref*{#1:##1}}}}
\newthm{thm}{Theorem}
\newthm{lemma}{Lemma}
\newthm{prop}{Proposition}
\newthm{cor}{Corollary}
\newthm{conj}{Conjecture}
\newthm{cons}{Consequence}
\newthm{fact}{Fact}
\newthm{obs}{Observation}
\newthm{guess}{Guess}
\newthm{algo}{Algorithm}
\theoremstyle{definition}
\newthm{defn}{Definition}
\newthm{remark}{Remark}
\newthm{example}{Example}
\newthm{question}{Question}
\newthm{notation}{Notation}

\newcommand{\Section}[1]{\hyperref[sec:#1]{Section~\ref*{sec:#1}}}
\newcommand{\Table}[1]{\hyperref[tab:#1]{Table~\ref*{tab:#1}}}
\newcommand{\eqn}[1]{\hyperref[eqn:#1]{(\ref*{eqn:#1})}}
\newcommand{\Figure}[1]{\hyperref[fig:#1]{Figure~\ref*{fig:#1}}}

\makeatletter
\def\namedlabel#1#2{\begingroup
    #2%
    \def\@currentlabel{#2}%
    \phantomsection\label{#1}\endgroup
}
\makeatother

\DeclareMathOperator{\GL}{GL}

\DeclareMathOperator{\Gr}{Gr}
\DeclareMathOperator{\Fl}{Fl}

\DeclareMathOperator{\QH}{QH}
\DeclareMathOperator{\QK}{QK}
\DeclareMathOperator{\K}{K}

\DeclareMathOperator{\HH}{H}
\DeclareMathOperator{\pt}{pt}

\DeclareMathOperator{\codim}{codim}

\DeclareMathOperator{\Rep}{Rep}

\def\poly{{\mathrm{poly}}}
\def\loc{{\mathrm{loc}}}

\DeclareMathOperator{\supp}{supp}

\DeclareMathOperator{\ch}{\ch}

\usepackage{bm}

\DeclarePairedDelimiter{\angles}{\langle}{\rangle}

\newcommand{\C}{{\mathbb C}}

\newcommand{\Z}{{\mathbb Z}}

\newcommand{\cF}{{\mathcal F}}

\newcommand{\cO}{{\mathcal O}}

\newcommand{\cQ}{{\mathcal Q}}

\newcommand{\cS}{{\mathcal S}}
\newcommand{\fS}{{\mathfrak S}}

\newcommand{\egw}[2]{\angles{ #1 }^T_{#2}}

\newcommand{\la}{{\lambda}}

\DeclareMathOperator{\ev}{ev}
\DeclareMathOperator{\Jac}{Jac}

\newcommand{\wb}{\overline}

\newcommand{\ignore}[1]{}

\newcommand{\Mb}{\wb{\mathcal M}}

\newcommand{\M}{\overline{M}}

\usepackage[margin=1in]{geometry}

\begin{document}

\title
[QK Whitney]
{Quantum K Whitney relations for partial flag varieties}

\author{Wei Gu}
\address{Max-Planck-Institute f\"{u}r Mathematik, Vivatsgasse 7, D-53111 Bonn, Germany}
\email{guwei@mpim-bonn.mpg.de}

\author{Leonardo C.~Mihalcea}
\address{
Department of Mathematics, 
225 Stanger Street, McBryde Hall,
Virginia Tech University, 
Blacksburg, VA 24061
USA
}
\email{lmihalce@vt.edu}

\author{Eric Sharpe}
\address{Department of Physics MC 0435, 850 West Campus Drive,
Virginia Tech University, Blacksburg VA  24061
USA}
\email{ersharpe@vt.edu}

\author{Weihong Xu}
\address{Division of Physics, Mathematics,  and Astronomy,
Caltech, 1200 E. California Blvd.,
Pasadena CA 91125  USA}
\email{weihong@caltech.edu}

\author{Hao Zhang}
\address{Department of Physics MC 0435, 850 West Campus Drive,
Virginia Tech University, Blacksburg VA  24061
USA}
\email{hzhang96@vt.edu}

\author{Hao Zou}
\address{Beijing Institute of Mathematical Sciences and Applications,
Beijing 101408, China}
\address{Yau Mathematical Sciences Center,
Tsinghua University,
Beijing 100084, China}
\email{hzou@vt.edu}

\thanks{WG was partially supported by NSF grant PHY-1720321. LM was partially supported by NSF grant DMS-2152294 and a Simons Collaboration Grant. 
ES was partially supported by NSF grant PHY-2310588. HZ was partially supported by the China Postdoctoral Science Foundation with grant No. 2022M720509}

\subjclass[2020]{Primary 14M15, 14N35, 81T60; Secondary 05E05}
\keywords{Quantum K theory, partial flag varieties, incidence varieties,
Whitney relations, Coulomb branch.}

\begin{abstract} In a recent paper, we stated conjectural presentations 
for the equivariant quantum K ring of partial
flag varieties,
motivated by physics considerations. In this companion paper, we analyze these presentations mathematically. We start by proving a Nakayama type result for quantum K theory: if the conjectured set 
of relations deforms a complete set of relations of the classical K theory ring, 
then they must form a complete set of relations for the quantum K ring. We prove
the conjectured presentation in the case of
the incidence varieties, and we show that if a quantum K divisor axiom holds (as conjectured by Buch and Mihalcea), then 
the conjectured presentation also holds for the complete flag variety. Finally, we 
briefly revisit the change of variables relating the mathematics and physics presentations.
\end{abstract}

\maketitle

\setcounter{tocdepth}{1}
\tableofcontents

\section{Introduction}\label{sec:intro}

In \cite{Gu:2020zpg}, conjectural presentations by generators 
and relations for 
the quantum \(\K\)-theory rings of the ordinary Grassmanians were stated. These presentations 
come in two flavors: a ``Coulomb branch" presentation which arises in physics 
as the critical locus of a certain one-loop twisted superpotential associated 
to a gauged linear sigma model (GLSM) (cf.~\S \ref{sec:physics} below), 
and a ``quantum K Whitney" presentation which arises in mathematics 
from a quantization of the classical Whitney relations. 
A mathematical proof of these presentations was given in 
\cite{gu2022quantum}, in the more general equivariant setting; an alternative proof was recently given in \cite{sinha2024quantum}. Continuing the work in \cite{gu2022quantum}, 
we recently conjectured in \cite{gu2023quantum} an extension of these presentations from 
Grassmannians to any partial flag variety, and we analyzed these conjectures 
from the physics point of view. 

Our goal in this paper is to investigate these presentations 
mathematically.
We prove that our conjectures hold in the case of incidence varieties 
$\Fl(1,n-1;n)$ which parametrize flags of vector spaces $F_1 \subset F_2 \subset \C^n$ with 
$\dim F_1 = 1 $ and \(\dim F_2=n-1\). The conjectural presentation also holds for 
the complete flag varieties, conditional on the validity of an unpublished conjecture by 
Buch and Mihalcea about a divisor axiom in quantum 
\(\K\)-theory; cf.~\Cref{conj:full} below. In the upcoming paper \cite{AHMOX}, 
we will use these results
to deduce the presentation for the quantum K ring of any partial flag manifold, by
using Kato's push-forward in quantum K theory \cite{kato:partial}, see also 
\cite{kouno.et.al:parabolic}. A very recent different proof, contemporaneous
with the latter work, has been announced in \cites{huq2024quantum}, 
using a quantum \(\K\) version of the abelian/non-abelian correspondence.

A key feature of all these presentations 
is that a complete set of relations is obtained by deforming each 
of the relations in the classical Whitney presentations; no other relations are needed. 
For quantum cohomology - a graded ring - this statement goes back to 
Siebert and Tian \cite{siebert.tian:on}, and uses a graded version of Nakayama's lemma.
The quantum K-theory ring is not graded, and a more careful analysis is needed, which requires
working over completed rings, and checking a module finiteness hypothesis. 
In this paper we revisit the treatment from \cite{gu2022quantum} on this problem (cf.~ \Cref{prop:Nakiso}), and we find natural hypotheses under which a finiteness condition needed in the 
K-theoretic version of Siebert and Tian criterion holds. The main
statements are found in the Appendix \ref{app:fg}. We believe this will be valuable  
in applications to more general situations.

Finally, we also take the opportunity to clarify the relation between the presentations
arising from the ``QK Whitney'' and the ``Coulomb branch'' relations. In this paper
we use a set of variables based on K-theoretic Chern
roots - different from the choice in \cite{gu2022quantum} - which makes it clear that the two presentations, while not equal, only differ by a change of variables obtained by dividing by some factors of the form $1-q_j$; see \Cref{prop:iso-Coulomb-Whitney} and 
\Cref{remark:general}. 

As explained in \cite{gu2023quantum},
our Whitney presentation may be viewed as a \(\K\)-theoretic analogue of a 
presentation obtained by Gu and Kalashnikov \cite{GuKa} 
of the quantum cohomology ring of quiver flag varieties.
A presentation of the equivariant quantum K ring of the complete flag varieties $\Fl(n)$
was recently proved by Maeno, Naito and Sagaki \cite{maeno.naito.sagaki:QKideal}, and  
it is related to the 
Toda lattice presentation from quantum cohomology \cites{givental.kim, kim.toda}. 
In an upcoming paper \cite{AHMOX}, we will rederive this Toda-type presentation 
as a consequence of our Whitney presentation. At its root, this reflects the fact that the Whitney presentation encodes quantum products between (tautological) {\em vector} bundles and their quotients, while the Toda presentation only deals with the quotient line bundles. Relations similar to those from \cites{maeno.naito.sagaki:QKideal,givental.kim} 
also appear in \cite{givental.lee:quantum,anderson.et.al:finiteI}, 
in relation to the finite-difference Toda lattice; in \cite{koroteev}, in relation to the study
of the quasimap quantum \(\K\)-theory of the cotangent bundle of $\Fl(n)$ and the 
Bethe Ansatz (see also \cite{Gorbounov:2014}); 
and in \cite{ikeda.iwao.maeno}, in relation to the Peterson isomorphism in quantum \(\K\)-theory.

\subsection{Statement of results}
We provide next a more precise account of our results.
Let \(X=\Fl(r_1,\dots,r_k;n)\) be a partial flag variety, equipped 
with the flag of tautological vector bundles 
\[
    0=\cS_{0}\subset\cS_{1}\subset\ldots \subset \cS_{k}\subset\cS_{{k+1}}=\C^n,
\]  
where \(\cS_{j}\) has rank \(r_j\). The variety \(X\) is a homogeneous space for the left action 
by $G:=\GL_n(\C)$, and we denote by $T \subset G$ the maximal torus consisting 
of diagonal matrices. The $T$-equivariant quantum K ring $\QK_T(X)$, defined
by Givental and Lee \cites{givental:onwdvv,lee:QK}, is a deformation of the Grothendieck ring $\K_T(X)$ of $T$-equivariant vector bundles on $X$. Additively,
$\QK_T(X) = \K_T(X) \otimes_{\K_T(\pt)} \K_T(\pt)[\![q]\!]$ where 
$\K_T(\pt) = \Rep(T)$ is the 
representation ring of $T$ and $\K_T(\pt)[\![q]\!]:=\K_T(\pt)[\![q_1, \ldots , q_k]\!]$ is a power series 
ring in the sequence of quantum parameters $(q_i)$ indexed by a 
basis of $\HH_2(X)$. This $\K_T(\pt)[\![q]\!]$-module is
equipped with a quantum product $\star$ which gives \(\QK_T(X)\) the structure of 
a commutative, associative, $\K_T(\pt)[\![q]\!]$-algebra. For $E \to X$ an equivariant
vector bundle of rank $\mathrm{rk}~E$ we denote by 
\[ \lambda_y(E) := 1 + y [E] + \ldots + y^{\mathrm{rk}\,E} [\wedge^{\mathrm{rk}\,E} E] \quad \in \K_T(X)[y] \]
the Hirzerbruch $\lambda_y$ class of $E$. This class is multiplicative for short exact sequences. {In an abuse of notation, we sometimes write \(E\) for the class \([E]\) in \(\K_T(X)\).}

The following is our main conjecture, also stated in the companion paper \cite{gu2023quantum}. 
\begin{conj}\label{conj:lambda_y} For  \(j=1,\dots, k\), the following relations hold in $\QK_T(X)$:
        \begin{equation}\label{eqn:lambda_y_rel_intro}
            \lambda_y(\cS_{j})\star\lambda_y(\cS_{{j+1}}/\cS_{j})=\lambda_y(\cS_{{j+1}})-y^{r_{j+1}-r_j}\frac{q_j}{1-q_j}\det(\cS_{{j+1}}/\cS_{j})\star(\lambda_y(\cS_{j})-\lambda_y(\cS_{{j-1}})).
        \end{equation} 
Here, $\lambda_y(\cS_{k+1})=\lambda_y(\C^n) = \prod_{i=1}^n (1+yT_i)$, where $T_i \in \K_T(\pt)$ are given by
the decomposition of $\C^n$ into one dimensional $T$-modules.
\end{conj}
{Despite finiteness results for \(\QK_T(X)\) proved in \cite{anderson2022finite} and \cite{kato:loop}*{Cor. 4.16}, we note that inverting \(1-q_j\) is essential, see \Cref{rmk:poly}.}

If $k=1$, i.e., if $X=\Gr(r,n)$, these relations were conjectured in \cite{Gu:2020zpg}, 
and proved in \cite{gu2022quantum}.  
Specializing all $q_j$ to \(0\) recovers the usual Whitney relations in $\K_T(X)$ obtained
from the short exact sequences 
\[
    0 \to \cS_{j} \to \cS_{{j+1}} \to \cS_{{j+1}}/\cS_{j} \to 0.
\] 
These relations may be formalized by introducing abstract variables
$X^{(j)}=(X^{(j)}_1,\dots,X^{(j)}_{r_j}) $ and $Y^{(j)}=(Y^{(j)}_1,\dots,
{ Y^{(j)}_{r_{j+1}-r_j}})$
for the exponentials of the Chern roots
of $\cS_j$ and $\cS_{j+1}/\cS_j$, respectively; see \eqref{defn:Iq} below. 
Set 
\[
    S\coloneqq \K_T(\pt)[e_1(X^{(j)}),\dots, e_{r_j}(X^{(j)}),e_1(Y^{(j)}),\dots,e_{r_{j+1}-r_j}(Y^{(j)}), j=1,\dots, k], 
\]
and let \(I_q\subset S[\![q]\!]\) be
the ideal {generated by the coefficients of \(y\) in \eqref{eqn:lambda_y_rel_intro}}. Our first result is the following, cf.~\Cref{thm:all_relations}.

\begin{thm}\label{thm:complete-intro}
Assume \Cref{conj:lambda_y} holds. Then the relations
\eqn{lambda_y_rel_intro} form a complete set of relations,
i.e., there is an isomorphism 
\[ \Psi:  {S[\![q]\!]}/I_q\to \QK_T(X) \]
of \(\K_T(\pt)[\![q]\!]\)-algebras sending \(e_\ell(X^{(j)})\) to \(\wedge^\ell(\cS_{j})\) and \(e_\ell(Y^{(j)})\) to \(\wedge^\ell(\cS_{{j+1}}/\cS_{j})\).
\end{thm}
The proof of this theorem follows a strategy developed in \cite{gu2022quantum},
generalizing a classical result by Siebert and Tian \cite{siebert.tian}
(see also \cite{fulton.pandh:notes}) about the quantum cohomology ring. 
Roughly speaking, in order to find the ideal of quantum relations, it suffices to find a presentation
of the classical ring, then quantize each relation in this presentation. The quantum
cohomology version follows easily from a graded version of Nakayama's lemma. The
version needed for quantum \(\K\)-theory, is more subtle, and is 
given in \Cref{prop:Nakiso} below. A key hypothesis needed in this proposition is 
that the claimed presentation is finitely generated as a $\K_T(\pt)[\![q]\!]$-module.
We prove this in Appendix \ref{app:fg} (see~\Cref{prop:KNak}) in a rather general 
context about modules over formal power series rings; {it is closely 
related to statements from \cite{eisenbud:CAbook}.}
This method might be of use for proving presentations for more general cases.
In this paper we prove \Conjecture{lambda_y} for the incidence varieties $\Fl(1,n-1;n)$, cf. \Theorem{rels} and \Theorem{incidence}.
We denote by $q_1, q_2$ the quantum parameters corresponding to the Schubert 
curves indexed by the simple reflections $s_1=(1,2), s_{n-1}=(n-1, n)$, respectively.

\begin{thm}\label{thm:main-intro1} \Cref{conj:lambda_y} holds for $\Fl(1,n-1;n)$. More
explicitly, the following relations hold in $\QK_T(\Fl(1,n-1;n))$, and they form a complete set of relations:
\begin{equation*}
            \lambda_y(\cS_{1})\star\lambda_y({\cS_{2}}/\cS_{1})=\lambda_y({\cS_{2}})-y^{n-2}\frac{q_1}{1-q_1}\det({\cS_{2}}/\cS_{1})\star(\lambda_y(\cS_{1})-1) \/;
        \end{equation*} 
\begin{equation*}
            \lambda_y({\cS_{2}})\star\lambda_y(\C^n/{\cS_{2}})=\lambda_y(\C^n)-y\frac{q_2}{1-q_2}\det(\C^n/{\cS_2})\star(\lambda_y({\cS_{2}})-\lambda_y(\cS_{1})) \/.
        \end{equation*} 

Here, $\lambda_y(\C^n) = \prod_{i=1}^n (1+yT_i)$, where $T_i \in \K_T(\pt)$ are given by the decomposition of $\C^n$ into one dimensional $T$-modules. 
\end{thm}
The proof of \Cref{thm:main-intro1} relies on a recent result of Xu 
\cite{xu2021quantum}, which proves a conjecture 
by Buch and Mihalcea in the case of incidence varieties 
\(\Fl(1,n-1;n)\); cf.~\Cref{conj:full} below. 
This conjecture, which may be thought of as a divisor axiom in 
quantum \(\K\)-theory, {gives a formula for computing $3$-point (equivariant) 
\(\K\)-theoretic Gromov-Witten 
invariants 
when one of the insertions is a line bundle class such as \(\det(\cS_j)\). Our strategy is to reduce the relations in \(\QK_T(\Fl(1,n-1;n))\) into a form 
involving no other quantum multiplications except for those of the form \(\det(\cS_{j})\star a\), turn these relations into identities involving $3$-pointed (equivariant) 
\(\K\)-theoretic Gromov-Witten 
invariants, apply the aforementioned formula, and do some computations in ordinary (equivariant) \(\K\)-theory.}
In addition, Xu proved in \cite{xu2021quantum} that the Schubert classes
in $\QK_T(\Fl(1,n-1;n))$ are generated by $\cS_1$ and $\det (\cS_2)$
over $\K_T(\pt)[q]$. As a consequence, one may 
replace the power series ring $\K_T(\pt)[\![q]\!]$ by the localized ring
$\K_T(\pt)[q]_{1 + \langle q \rangle}$. 

Aside from the incidence varieties, the quantum K divisor axiom (\Cref{conj:full}) is known for cominuscule Grassmannians
\cites{buch.m:qk,chaput.perrin:rationality}. A natural question is whether more cases of \Cref{conj:lambda_y} may be proved assuming the validity of \Cref{conj:full}. This leads to 
our next result, cf. \Corollary{fln}.

\begin{thm} Assume the quantum K divisor axiom (i.e., \Cref{conj:full}) holds for the complete flag variety $\Fl(n)$. Then 
the relations \eqref{eqn:lambda_y_rel_intro} hold in $\QK_T(\Fl(n))$. Equivalently, \Cref{conj:full} implies 
\Cref{conj:lambda_y} for \(\Fl(n)\).\end{thm}
The proof of this statement employs the technique of 
curve neighborhoods developed in \cite{buch.m:nbhds}, see also
\cite{BCMP:qkfin}. For an effective degree $d \in \HH_2(X)$, we utilize an interpretation 
of the curve neighborhood $a[d]$ of an element 
$a \in \K_T(X)$ 
{as $\partial_{z_d}(a)$, the result of applying an iterated Demazure operator,}
where $z_d$ is a permutation defined in \cite{buch.m:nbhds}. This interpretation,
also utilized in \cite{li.mihalcea}, allows us to 
compute the curve neighborhoods
of the exterior powers of tautological bundles 
$\wedge^\ell \cS_i$; see \Cref{lemma:dem-bundles}. 
The key {calculation} is \Cref{cor:partialzd}, which establishes an equality between 
{curve neighborhoods of exterior powers of adjacent tautological bundles, when the degrees differ by a simple (co)root.}

{\em Acknowledgments.}
We thank Anders Buch, Linda Chen, Elana Kalashnikov, Peter Koroteev, Y.P.~Lee, 
and Henry Liu for helpful discussions. Special thanks are due to Prof.~Satoshi 
Naito for many inspiring discussions, 
and for pointing out some references utilized in the Appendix. Finally, 
we thank the Simons Center for Geometry and Physics for supporting the excellent workshop GLSM@30, which made possible many stimulating discussions.

\section{Equivariant K-theory}\label{sec:preliminaries} 

\subsection{Preliminaries}
In this section we recall some basic facts about the equivariant 
K-theory of a variety with a group action. For an introduction to equivariant K theory, and more details, see \cite{chriss2009representation}. 

Let $X$ be a smooth projective variety with an action of a linear algebraic group $G$. The equivariant K theory ring $\K_G(X)$ is the Grothendieck ring generated by symbols $[E]$, where $E \to X$ is a $G$-equivariant vector bundle, modulo the relations $[E]=[E']+[E'']$ for any short exact sequence $0 \to E' \to E \to E'' \to 0$ of equivariant vector bundles. The additive ring structure is given by direct sum, and the multiplication is given by tensor products of vector bundles. 
Since $X$ is smooth, any $G$-linearized coherent sheaf has a finite resolution by (equivariant) vector bundles, and the ring $\K_G(X)$ coincides with the Grothendieck group of $G$-linearized coherent sheaves on $X$. In particular, any 
$G$-linearized coherent sheaf $\mathcal{F}$ on $X$ determines a class $[\mathcal{F}] \in \K_G(X)$. 
As an important special case, if $\Omega \subset X$ is a $G$-stable subscheme, then its structure sheaf 
determines a class $[\cO_\Omega] \in \K_G(X)$. 
We shall abuse notation and sometimes write \(\cF\) or $\cO_\Omega$ for the corresponding classes 
$[\cF]$ and $[\cO_\Omega]$ in $\K_G(X)$. 

The ring $\K_G(X)$ is an algebra over $\K_G(\pt) = \Rep(G)$, the representation ring of 
$G$. If $G=T$ is a complex torus, then this is the Laurent polynomial ring 
$\K_T(\pt) = \Z[T_1^{\pm 1}, \ldots, T_n^{\pm 1}]$ where 
$T_i:=\C_{t_i}$ are characters corresponding to a basis of the 
Lie algebra of $T$. 

Let $E \to X$ be an equivariant
vector bundle of rank $\mathrm{rk}\,E$. The (Hirzebruch) $\lambda_y$ class is defined as 
\[ \lambda_y(E) := 1 + y E + \ldots + y^{\mathrm{rk}\,E} \wedge^{\mathrm{rk}\,E} E \quad \in \K_T(X)[y]. \] 
This class was introduced by Hirzebruch \cite{hirzebruch:topological} in relation 
to the Grothendieck--Riemann--Roch theorem. 
The $\lambda_y$ class is multiplicative with respect to short exact sequences, i.e., if 
$0 \to E' \to E \to E'' \to 0$ is such a sequence of vector bundles, then 
\[ \lambda_y(E) = \lambda_y(E') \cdot \lambda_y(E''). \] 
This is part of the $\lambda$-ring structure of $\K_T(X)$, see e.g. \cite{nielsen:diag}, referring to \cite{SGA6}. 

A particular case of this construction is when $V$ is a (complex) vector space with an action 
of a complex torus $T$. The  $\lambda_y$ class of $V$ is the element 
$\lambda_y(V) =\sum_{i \ge 0} y^i \wedge^i V \in \K_T(\pt)[y]$. Since \(V\) decomposes into $1$-dimensional $T$-representations: $V = \bigoplus_i \C_{\mu_i}$, 
it follows from the multiplicative property of $\lambda_y$ classes that 
\(\lambda_y(V) = \prod_i (1+y\,\C_{\mu_i})\). 

Since $X$ is proper, we can push the class of a sheaf forward to the point. This is given by the sheaf Euler 
characteristic, or, equivalently, the virtual representation 
\[\chi_X^T(\mathcal{F}) 
:= \sum_i (-1)^i \HH^i(X, \mathcal{F})\quad\in \K_T(\pt) \/. \] 

\subsection{(Equivariant) \(\K\)-theory of flag varieties}\label{sec:kflag} The partial flag variety
$$X:=\Fl(r_1, \ldots, r_k;n)$$ parametrizes flags of vector spaces 
$F_{1} \subset F_{2} \subset \ldots \subset F_{k} \subset \C^n$
with $\dim F_{i} = r_i$ for $1 \le i \le k$. 
It is a projective manifold and admits a transitive action 
of $G=\GL_n(\C)$. We denote by $T$ the maximal torus in $G$ consisting of diagonal matrices.

Let $S_n$ be the symmetric group in $n$ letters, and let $S_{r_1, \ldots, r_k} \le S_n$ be 
the subgroup generated by simple reflections $s_i = (i,i+1)$ where $i \notin \{ r_1, \ldots, r_k\}$. 
{For $w \in S_n$, its length $\ell(w)$ is defined to be the minimal number of simple reflections with product $w$.}
Denote by $\ell:S_n \to \mathbb{N}$ the length function, and by  
$S^{r_1, \ldots, r_k}$ the set of minimal length representatives of 
$S_n/S_{r_1, \ldots, r_k}$. This consists
of permutations $w \in S_n$ which have descents at most at positions 
$r_1, \ldots, r_k$, i.e., $w(r_j+1)< \ldots <w(r_{j+1})$, for $j=0, \ldots, k$, with the convention that $r_0=0$ and $r_{k+1} =n$. 

The \(T\)-fixed points $e_w \in X$ are indexed by the permutations 
$w \in S_n^{r_1, \ldots, r_k}$. {Let \(B,\ B^-\subset G\) be the Borel subgroups of upper and lower triangular matrices, respectively.}
For each \(T\)-fixed point, {the \(B\)-stable Schubert variety \(X_w=\overline{B.e_w}\) and \(B^-\)-stable Schubert variety \(X^w=\overline{B^-.e_w}\) are closures of the \(B\) and \(B^-\) orbits in \(X\), respectively. We have \(\dim X_w=\codim X^w=\ell(w)\). Let \(\cO_w=[\cO_{X_w}]\) and \(\cO^w=[\cO_{X^w}]\) be the classes in \(\K_T(X)\) determined by the structure sheaves of \(X_w\) and \(X^w\), respectively. The ring \(\K_T(X)\) is a free module over 
$\K_T(\pt)$ with a basis given by these Schubert classes:
\[ \K_T(X) = \bigoplus_{w \in S_n^{r_1,\ldots, r_k}} \K_T(\pt) \cO_w = \bigoplus_{w \in S_n^{r_1,\ldots, r_k}} \K_T(\pt) \cO^w \/. \]} 

Denote by 
$0=\cS_{0}\subset\cS_{1}\subset\dots \subset \cS_{k}\subset\cS_{{k+1}}=\C^n$  
the flag of tautological vector bundles on $X$, where \(\cS_{j}\) has rank \(r_j\). 
Since we could not find a precise reference,
we will take the opportunity to outline a proof 
for a  (folklore) presentation by generators and relations 
of $\K_T(\Fl(r_1, \ldots, r_k;n))$. The relations 
\begin{equation*}
    \lambda_y(\cS_{j})\cdot\lambda_y(\cS_{{j+1}}/\cS_{j})=\lambda_y(\cS_{{j+1}}),\quad j=1,\dots,k
\end{equation*}
arise from the Whitney relations
applied to the exact sequences 
\[
    0 \to \cS_{{j}} \to \cS_{j+1} \to \cS_{j+1}/\cS_{{j}} \to 0, \quad j=1,\dots,k,
\] 
and are 
specializations with $q_j =0, j=1,\dots,k$ of the relations from
\Cref{conj:lambda_y}. This presentation is related
to well-known presentations such as that in
\cite{lascoux:anneau}*{\S 7}.

More precisely, let 
\[
    {X}^{(j)}=(X^{(j)}_1,\dots,X^{(j)}_{r_j})\text{ and }{Y}^{(j)}=(Y^{(j)}_1,\dots,
{Y^{(j)}_{r_{j+1}-r_j}})
\] 
denote formal variables for \(j=1,\dots,k\).
Let \(X^{(k+1)}\coloneqq (T_1,\dots,T_n)\) be the equivariant parameters in \(\K^T(\pt)\). Geometrically, the variables $X_i^{(j)}$ and $Y_s^{(j)}$ arise from the splitting principle:
\[ \lambda_y(\cS_j) = \prod_i (1+y X_i^{(j)}) \/, \quad  \lambda_y(\cS_{j+1}/\cS_{j}) = \prod_s (1+y Y_s^{(j)}), \]
i.e., they are the K-theoretic Chern roots of 
$\cS_{j}$ and $\cS_{j+1}/\cS_{{j}}$, respectively. Let \(e_\ell({X}^{(j)})\) and \(e_\ell({Y}^{(j)})\) be the \(\ell\)-th elementary symmetric polynomials in \({X}^{(j)}\) and \({Y}^{(j)}\), respectively. 
Denote by $S$ the (Laurent) polynomial ring 
\[ 
 S\coloneqq\K_T(\pt)[e_i(X^{(j)}), e_s(Y^{(j)}); 1 \le j \le k, 1 \le i \le r_j, 1 \le s \le {r_{j+1}-r_j}]
\]
and define the ideal $I \subset S$ generated by 
\begin{equation}\label{eqn:rel_classical}
    \sum_{i+s=\ell} e_i(X^{(j)}) e_s(Y^{(j)}) - e_\ell(X^{(j+1)})\/, \quad 1 \le j \le k \/. 
\end{equation}

\begin{prop}\label{prop:classical-Whitney} There is an isomorphism of $\K_T(\pt)$-algebras:
\[ \Psi: S/I \to \K_T(X) \]
sending $e_i(X^{(j)}) \mapsto \wedge^i \cS_{j}$ and 
$e_i(Y^{(j)}) \mapsto \wedge^i {(\cS_{{j+1}}/\cS_{{j}})}$. 
\end{prop}
\begin{proof} Denote the conjectured presentation ring by $A$. The K theoretic Whitney relations
imply that $\lambda_y(\cS_j) \cdot \lambda_y(\cS_{j+1}/\cS_j) = \lambda_y(\cS_{j+1})$.
Then the geometric interpretation of the variables $X^{(j)}, Y^{(j)}$ in terms of the splitting principle 
before the theorem implies that $\Psi$ is a well-defined $\K_T(\pt)$-algebra homomorphism. 

To prove surjectivity of $\Psi$, we first consider the case when $X=\Fl(n)$ is the full flag variety, and 
we utilize the theory of double Grothendieck polynomials
\cite{fulton.lascoux,buch:grothendieck}. {It was proved in \cite{buch:grothendieck}*{Thm. 2.1} that 
each Schubert class in $\K_T(X)$ may be written as a (double Grothendieck) polynomial in  
\begin{equation*}\label{E:vars} 
1-(\C^n/\cS_{n-1})^{-1}, 1- (\cS_{n-1}/\cS_{n-2})^{-1}, \ldots , 1 - (\cS_2/\cS_1)^{-1}, 1-(\cS_1)^{-1} 
\end{equation*}
with coefficients in $\K_T(\pt)$. Note that in $\K_T(X)$, 
\begin{align*}
    (\cS_{i}/\cS_{i-1})^{-1}
    =&\det(\C^n)^{-1}\cdot\C^n/\cS_{n-1}\cdots{\cS_{i+1}/\cS_{i}}\cdot\cS_{i-1}/\cS_{i-2}\cdots\cS_2/\cS_1\cdot\cS_1 
\end{align*}
for \(i=1,\dots,n\).
Therefore, each Schubert class may be written 
as a polynomial in variables \(\cS_{i}/\cS_{i-1}\) for \(i=1,\dots,n\) with coefficients in \(\K_T(\pt)\). This proves the surjectivity in this case. 

For partial flag varieties, consider the 
injective ring homomorphism given by pulling back via the natural projection 
$p: \Fl(n) \to \Fl(r_1, \ldots, r_k;n)$. The pullbacks of Schubert classes
and of the tautological bundles are 
\[ p^*\cO^w = \cO^w \/, \quad p^*(\wedge^\ell \cS_i) = \wedge^\ell \cS_{r_i} \/, \quad 
p^*(\wedge^\ell (\cS_i/\cS_{i-1})) = \wedge^\ell (\cS_{r_i}/\cS_{r_{i-1}})  \quad \/, \]
for any $w \in S_n^{r_1, \ldots, r_k}$, any $ 1 \le i \le k$, and any $\ell$. 
On the other hand, since $w \in S_n^{r_1, \ldots, r_k}$, the Schubert classes $p^*\cO^w$
may be written as (double Grothendieck) polynomials symmetric in each block of variables 
$1-(\cS_{r_i+1}/\cS_{r_i})^{-1}, \ldots, 1-(\cS_{r_{i+1}}/\cS_{r_{i+1}-1})^{-1}$, for $0 \le i \le k$, 
i.e., in the elementary symmetric functions 
$e_{\ell}((1-(\cS_{r_i+1}/\cS_{r_i})^{-1}, \ldots, 1-(\cS_{r_{i+1}}/\cS_{r_{i+1}-1})^{-1}))$
in these sets of variables. Each such elementary symmetric function may be further expanded  
as a $\Z$-linear combination of terms of the form
\begin{equation*}\label{E:fractions} \frac{e_{s}(\cS_{r_i+1}/\cS_{r_i}, \ldots, \cS_{r_{i+1}}/\cS_{r_{i+1}-1})}{\cS_{r_i+1}/\cS_{r_i}\cdot \ldots \cdot \cS_{r_{i+1}}/\cS_{r_{i+1}-1}} = \frac{\wedge^s(\cS_{r_{i+1}}/\cS_{r_i})}{\det (\cS_{r_{i+1}}) / \det (\cS_{r_i})} \/. \end{equation*}
Observe that 
$\det \C^n = \prod_{i=1}^{k+1} {\det} (\cS_{r_i}/\cS_{r_{i-1}}) {= \det(\cS_{r_{j}}) \prod_{i=j+1}^{k+1}\det(\cS_{r_i}/\cS_{r_{i-1}})}$.
Therefore,
\[ \det( \cS_{r_{j}})^{-1} = (\det \C^n)^{-1} {\prod_{i=j+1}^{k+1}\det(\cS_{r_i}/\cS_{r_{i-1}})}, \quad j=1,\dots, k. \]
We have shown that the pullbacks of Schubert classes $p^*(\cO^w)$ are polynomials in the pullbacks of the tautological bundles and their quotients, and 
we deduce that $\Psi$ is surjective for partial flag manifolds.}

Injectivity holds because $\K_T(\pt)$ is an integral domain and both
$A$ and $\K_T(X)$ have the same rank as free modules over $\K_T(\pt)$. 
To see the latter, consider the ring 
\[ A':= \Z[T_1, \ldots, T_n][e_i(X^j), e_s(Y^j): 1 \le j \le k, 1 \le i \le r_j, 1 \le s \le 
{r_{j+1}-r_j}]/I' , \]
where \(I'\subseteq A'\) is the ideal generated by \eqn{rel_classical}.
It is classically known that $\Z[T_1, \ldots, T_n]$-algebra $A'$ is isomorphic to
the equivariant cohomology algebra 
$\HH^*_T(X)$, with $e_i(X^{(j)})$ being sent to the equivariant Chern class $c_i^T(\cS_j)$ and 
$e_i(Y^{(j)})$ to the equivariant Chern class $c_i^T(\cS_{j+1}/\cS_j)$. (This follows from 
example by realizing the partial flag variety as a tower of Grassmann bundles, then using a description 
of the cohomology of the latter as in {\cite{fulton:IT}*{Example 14.6.6}}.) In particular, $A'$ is a free $\Z[T_1, \ldots, T_n]$-algebra
of rank equal to the number of Schubert classes in $X$. Then $A= A' \otimes_{\Z[T_1, \ldots, T_n]} \K_T(\pt)$
is a free $\K_T(\pt)$-module of the same rank.
\end{proof}

\subsection{Push-forward formulae of Schur bundles} Next, we recall some results about cohomology of 
Schur bundles on Grassmannians, which we will need later. Our main reference is Kapranov's paper \cite{kapranov:Gr}. A reference
for basic definitions of Schur bundles is Weyman's book \cite{weyman}.

Recall that if $X$ is a $T$-variety, $\pi:E \to X$ is any $T$-equivariant vector bundle of rank $e$, 
and $\lambda=(\lambda_1, \ldots, \lambda_r)$ 
is a partition with at most $e$ parts, the {\em Schur bundle} $\mathfrak{S}_\lambda(E)$ is a 
$T$-equivariant vector bundle over $X$. It has the property
that if $x \in X$ is a $T$-fixed point, the fibre $(\fS_\lambda(E))_x$ 
is the $T$-module with character the Schur function $s_\lambda$. For example,
if $\lambda = (1^k)$, then $\fS_{(1^k)}(E) = \wedge^k E$, and if $\lambda = (k)$ then
$\fS_{(k)}(E)=\mathrm{Sym}^k(E)$.

Consider a $T$-variety $X$ equipped with 
a $T$-equivariant
vector bundle $\mathcal{V}$ of rank $n$. Denote by 
$\pi: \mathbb{G}(r,\mathcal{V}) \to X$ the Grassmann bundle over $X$. 
It is equipped with a tautological sequence
$0 \to \underline{\cS} \to \pi^* \mathcal{V} \to \underline{\cQ} \to 0$ over 
$\mathbb{G}(r,\mathcal{V})$. The next result follows
from \cite[Prop. 2.2]{kapranov:Gr}, see also \cite[Prop. 3.2 and Cor. 3.3]{gu2022quantum}.
(Kapranov proved this when $X=\pt$; the relative version follows immediately
using that $\pi$ is a $T$-equivariant locally trivial fibration).

\begin{prop}[Kapranov]\label{prop:relativeBWB}
For any nonempty partition {$\lambda= (\lambda_1\geq \lambda_2 \geq \ldots \geq \lambda_r \geq 0)$ such that 
$\lambda_1 \le n-r$}, there are the following isomorphisms of $T$-equivariant vector bundles:
\begin{enumerate}
    \item For all $i \ge 0$, the higher direct images, $R^i \pi_* \fS_\lambda(\underline{\cS}) = 0$;
    \item \[ R^i \pi_*\fS_\lambda(\underline{\cS}^*)) = \begin{cases} \fS_\lambda(\mathcal{V}^*) & i=0 \\ 0 & i>0 \/; \end{cases}\]
    \item \[ R^i \pi_*\fS_\lambda(\underline{\cQ}) = \begin{cases} \fS_\lambda(\mathcal{V}) & i=0 \\ 0 & i>0 \/. \end{cases}\]
\end{enumerate}
\end{prop}

\section{(Equivariant) quantum K-theory of flag varieties}
In this section, we first recall the definition of the equivariant quantum K ring of a partial flag variety. We then proceed by proving a presentation of the quantum K ring by generators and relations of this ring,
generalizing the one from \Cref{prop:classical-Whitney}; cf.~\Cref{thm:all_relations}. 
At this time, the statement holds
under the assumption that certain generalized (quantum K) Whitney relations do hold in the quantum 
K ring; cf. \Cref{conj2:lambda_y}. In \Section{incidence}, we will show that this assumption is satisfied for incidence varieties,
and in \Section{fullflag} for the complete flag varieties, the latter under an extra assumption.

Throughout this section, we continue with the notation $G= \mathrm{GL}_n$ and 
$X=\Fl(r_1,\dots,r_k;n)$.  
\subsection{Preliminaries} 
{
    An effective degree is a \(k\)-tuple of nonnegative integers \(d=(d_1,\dots,d_k)\), which is identified with \(\sum_{i=1}^{k}d_i[X_{s_{r_i}}]\in\HH_2(X,\Z)\).
}
We write \(q^d\) for \(q_1^{d_1}\dots q_k^{d_k}\), where $q=(q_1, \ldots, q_k)$ is a sequence
of quantum parameters. 

We recall the definition of the \(T\)-equivariant (small) quantum 
K theory ring $\QK_T(X)$, following \cite{givental:onwdvv,lee:QK}: 
\begin{equation*}\label{eqn:QKTfree}
\QK_T(X) = \K_T(X) \otimes_{\K_T(\pt)} \K_T(\pt)[\![q]\!] 
\end{equation*}
is a free $\K_T(\pt)[\![q]\!]$-module
with a $\K_T(\pt)[\![q]\!]$-basis given by Schubert classes.
It is equipped with a commutative, associative 
product, denoted by $\star$, and
determined by the condition
\begin{equation}\label{eqn:def}
    (\!(\sigma_1\star\sigma_2,\sigma_3)\!)=\sum_d
        q^d \egw{\sigma_1,\sigma_2,\sigma_3}{d}
\end{equation}
for all \(\sigma_1,\sigma_2,\sigma_3\in \K_T(X)\), where  
\[
        (\!(\sigma_1,\sigma_2)\!)\coloneqq \sum_{d}q^d \egw{\sigma_1,\sigma_2}{d}
\]
is the quantum \(\K\)-metric and \(\egw{\sigma_1,\dots,\sigma_n}{d}\) are \(T\)-equivariant \(\K\)-theoretic Gromov--Witten invariants. We define these invariants next.

Let \(d\in \HH_2(X,\Z)_+\) be an effective degree and 
\(\Mb_{0,n}(X,d)\) 
be the Kontsevich moduli space parametrizing \(n\)-pointed, genus \(0\), degree \(d\) stable maps to \(X\). Let 
\[
\ev_1, \ev_2, \dots, \ev_n: \Mb_{0,n}(X,d)\to X
\] 
be evaluations at the \(n\) marked points. 
Given \(\sigma_1, \sigma_2,\dots, \sigma_n\in \K_T(X)\), we define the \(T\)-equivariant \(\K\)-theoretic Gromov--Witten invariant \[
\egw{\sigma_1,\sigma_2,\dots,\sigma_n}{d}\coloneqq\chi_{\M_{0,n}(X,d)}^T(\ev_1^*\sigma_1\cdot\ev_2^*\sigma_2\cdots\ev_n^*\sigma_n),
\] 
where \(\chi_Y^T: \K_T(Y)\to \K_T(\pt)\) is the pushforward to a point. We adopt the convention that when \(d\) is not effective, the invariant \(\egw{\sigma_1,\sigma_2,\dots,\sigma_n}{d}=0\). 

{\begin{remark} For $1 \le j \le k$, declare 
$\deg q_{j} = \deg ([X_{s_{r_j}}] \cap c_1(T_X)) =  r_{j+1}- r_{j-1}$.
For a multidegree
$d = (d_{r_1}, \ldots, d_{r_k})$, set $\deg(\cO^w \otimes q^d) = \ell(w) + \sum \deg(q_i) \cdot d_{r_i}$.
Together with the topological filtration on $\K_T(X)$, this equips $\QK_T(X)$ with a structure of a filtered ring;
see \cite[\S 5.1]{buch.m:qk}. The associated graded of this ring is $\QH^*_T(X)$, the (small) $T$-equivariant quantum cohomology of $X$, a free $\HH^*_T(\pt)[q]$-algebra of the same rank as $\QK_T(X)$. \end{remark} }

\subsection{A conjectural Whitney presentation for \(\QK_T(\Fl(r_1,\dots,r_k,n))\)}
As before, $X$ denotes \(\Fl(r_1,\dots,r_k,n)\), equipped with the tautological vector bundles
\[
    0=\cS_{0}\subset\cS_{1}\subset\dots \subset \cS_{k}\subset\cS_{{k+1}}=\C^n,
\]
where $\cS_j$ has rank $r_j$. Recall from \Section{intro} the following conjecture from \cite{gu2023quantum}.  
\begin{conj}\label{conj2:lambda_y} For  \(j=1,\dots, k\), the following relations hold in $\QK_T(X)$:
        \begin{equation}\label{eqn:lambda_y_rel}
            \lambda_y(\cS_{j})\star\lambda_y(\cS_{{j+1}}/\cS_{j})=\lambda_y(\cS_{{j+1}})-y^{r_{j+1}-r_j}\frac{q_j}{1-q_j}\det(\cS_{{j+1}}/\cS_{j})\star(\lambda_y(\cS_{j})-\lambda_y(\cS_{{j-1}})).
        \end{equation} 
\end{conj} 
{\em Assuming that this conjecture holds} we will state and prove a presentation by generators and relations of the ring $\QK_T(X)$. This conjecture was proved for Grassmannians \cite{gu2022quantum}, and we will verify it later for  the incidence varieties $X=\Fl(1,n-1;n)$ and the complete flag varieties $X=\Fl(n)\coloneqq\Fl(1,\dots,n-1;n)$, the latter under an additional assumption.

We start by transforming \eqref{eqn:lambda_y_rel} into an abstract presentation. As in \Section{kflag}, let 
\[
    {X}^{(j)}=(X^{(j)}_1,\dots,X^{(j)}_{r_j})\text{ and }{Y}^{(j)}=(Y^{(j)}_1,\dots, 
Y^{(j)}_{r_{j+1}-r_j})
\] 
denote formal variables for \(j=1,\dots,k\). Let \(X^{(k+1)}\coloneqq (T_1,\dots,T_n)\) be the equivariant parameters in \(\K_T(\pt)\). Let \(e_\ell({X}^{(j)})\) and \(e_\ell({Y}^{(j)})\) be the \(\ell\)-th elementary symmetric polynomials in \({X}^{(j)}\) and \({Y}^{(j)}\), respectively. 

\begin{defn}\label{defn:Iq}
    As before, 
    \[
        S=\K_T(\pt)[e_1(X^{(j)}),\dots, e_{r_j}(X^{(j)}),e_1(Y^{(j)}),\dots,e_{r_{j+1}-r_j}(Y^{(j)}), j=1,\dots, k].
    \]
    Let \(I_q\subset S[\![q]\!]=S[\![q_1,\dots,q_k]\!]\)
be the ideal generated by the coefficients of \(y\) in 
\begin{multline*}\label{eqn:qkrel}
    \prod_{\ell=1}^{r_j}(1+y X^{(j)}_\ell)\prod_{\ell=1}^{r_{j+1}-r_j}(1+y Y^{(j)}_\ell)-\prod_{\ell=1}^{r_{j+1}}(1+y X^{(j+1)}_\ell)\\+y^{{r_{j+1}-r_j}}\frac{q_j}{1-q_j}\prod_{\ell=1}^{r_{j+1}-r_j}Y^{(j)}_\ell\left(\prod_{\ell=1}^{r_j}(1+yX^{(j)}_\ell)-\prod_{\ell=1}^{r_{j-1}}(1+yX^{(j-1)}_\ell)\right),\ j=1,\dots, k.
\end{multline*}
\end{defn}
\begin{thm}\label{thm:all_relations}
        Assume \Cref{conj2:lambda_y} holds. Then there is an isomorphism of \(\K_T(\pt)[\![q]\!]\)-algebras
    \[
        \Psi: {S[\![q]\!]}/I_q\to \QK_T(X)
    \]
    sending \(e_\ell(X^{(j)})\) to \(\wedge^\ell(\cS_{j})\) and \(e_\ell(Y^{(j)})\) to \(\wedge^\ell(\cS_{{j+1}}/\cS_{j})\).
    \end{thm}
Different presentations have recently been obtained in \cite{maeno.naito.sagaki:QKideal} for the complete flag varieties, utilizing different methods. 

Note that the algebra homomorphism $\Psi$ is well-defined by \Cref{conj2:lambda_y}. 
The proof of \Cref{thm:all_relations} follows a method developed in \cite{gu2022quantum}, 
where \Cref{thm:all_relations} was 
proved for Grassmannians. 
For notions about completions we refer to \cite[Appendix A]{gu2022quantum} and \cite[Ch. 10]{AM:intro} for further details. The key fact we utilize is the following result proved in \cite[Prop. A.3]{gu2022quantum}.

\begin{prop}[Gu--Mihalcea--Sharpe--Zou]\label{prop:Nakiso} Let $A$ be a Noetherian integral domain, 
and let $\mathfrak{a} \subset A$ be an ideal.~Assume that $A$ is complete in the $\mathfrak{a}$-adic topology.
Let $M,N$ be finitely generated 
$A$-modules. 

Assume that the $A$-module $N$, and the $A/\mathfrak{a}$-module $N/\mathfrak{a}N$, 
are both free modules 
of the same rank $p< \infty$, and
that we are given  
an $A$-module homomorphism $f: M \to N$ such that
the induced $A/\mathfrak{a}$-module map 
$\overline{f}: M/\mathfrak{a} M \to N / \mathfrak{a}N$
is an isomorphism of $A/\mathfrak{a}$-modules.

Then $f$ is an isomorphism.  
\end{prop}

{A key hypothesis needed in this proposition is 
that the claimed presentation is finitely generated as a $\K_T(\pt)[\![q]\!]$-module.
This is proved in Appendix \ref{app:fg} (see~\Cref{prop:KNak}) in a rather general 
context about modules over formal power series rings.}

\begin{proof}[Proof of \Cref{thm:all_relations}] We use \Cref{prop:Nakiso}, with 
$M$ being the conjectured presentation on the left-hand side of \Cref{thm:all_relations}, 
\[ A= \K_T(\pt)[\![q_1, \ldots, q_k]\!], \quad \mathfrak{a}= \langle q_1, \ldots, q_k \rangle\subset A \/, \quad N= \QK_T(X)\/, \quad f= \Psi \/.\]  
Since \(N\) and $N/\mathfrak{a}N = \K_T(X)$ are both free modules of rank equal to the number of Schubert classes (over $A/\mathfrak{a}$, and over $A$, respectively), the hypotheses are satisfied for $N$ and 
$N/\mathfrak{a}N$.~\Cref{prop:classical-Whitney}
implies that the induced $A/\mathfrak{a}$-module map  map $\overline{f}: M/\mathfrak{a} M \to N / \mathfrak{a}N$
is an isomorphism of $A/\mathfrak{a}$-modules. Since \(R\coloneqq\K_T(\pt)\) and 
\(S\) are Noetherian, it then follows from \Proposition{KNak} that $M$ is a finitely generated $A$-module. Then the claim 
follows from \Cref{prop:Nakiso}.
\end{proof}

\section{Curve neighborhoods and some relations in \texorpdfstring{$\QK_T(\Fl(r_1,\dots,r_k;n))$}{the equivariant quantum K ring of partial flag varieties}}
Given that \Cref{thm:all_relations} depends on the validity \Cref{conj2:lambda_y}, our goal from now on will be to prove this conjecture in some special cases, notably
for the incidence varieties $\Fl(1,n-1;n)$ and -- under an extra assumption -- for the complete flag varieties $\Fl(n)\coloneqq\Fl(1,\dots,n-1;n)$. To this aim, we recall some geometric facts about {\em curve neighborhoods}, a
notion introduced in \cite{buch.m:nbhds} (see also \cite{BCMP:qkfin}) to help study spaces of rational curves
incident to Schubert varieties.

For an effective degree $d$, we define the degree \(d\) curve neighborhood of a (Schubert) variety \(\Omega\subseteq X\) to be 
 \begin{equation}\label{def:curvenbhd}
 \Gamma_d(\Omega)\coloneqq\ev_2\left(\ev_1^{-1}(\Omega)\right) \subset X \/.
 \end{equation}
 The degree \(d\) curve neighborhood of a class $\sigma \in \K_T(X)$ is defined by 
 \[
 \sigma[d]\coloneqq {\ev_2}_*\ev_1^*(\sigma) \in \K_T(X) \/. 
 \] 
 It was proved in \cite{BCMP:qkfin}*{Proposition 3.2} that if $\Omega$ is a Schubert variety, then
 \(\ev_2: \ev_1^{-1}(\Omega)\to\Gamma_d(\Omega)\) is a locally trivial fibration  with unirational fibers. It follows that 
    \[
        \cO_\Omega[d]=\cO_{\Gamma_d(\Omega)}.
    \]
    Moreover, it was proved in \cite{BCMP:qkfin,buch.m:nbhds} that
$\Gamma_d(\Omega)$ is again a Schubert variety. 
The following proposition follows from \cite[\S 2]{BCMP:qkfin} and \cite[\S 7]{buch.m:nbhds}. 
\begin{prop}[Buch--Chaput--Mihalcea--Perrin]
For any effective degree \(d\), a Schubert variety \(\Omega\subseteq X\), and \(j=1,\dots, k\), we have 
\begin{equation}\label{eqn:q=c1}
\egw{\cO_\Omega}{d}=\chi^T_X(\cO_{\Gamma_d(\Omega)})=1,
\end{equation}
\begin{equation}\label{eqn:q=c2}
\egw{\sigma, \cO_\Omega}{d}=\chi^T_X(\sigma\cdot\cO_{\Gamma_d(\Omega)}),
\end{equation}
and if \(\Omega\) is \(B\)-stable,
\begin{equation}\label{eqn:q=c2divisor}
\egw{\cO^{s_j}, \cO_\Omega}{d}=
\begin{cases}
    1 & d_j>0\\
    \chi^T_X(\cO^{s_j}\cdot\cO_\Omega) & d_j=0
.\end{cases}
\end{equation}
\end{prop}

{For \(i=1,\dots,k\), let \(p_j: X\to Y_j\coloneqq \Gr(r_j,n)\) be the equivariant projection. Abusing notation, we denote by \(\cS_j\) the tautological bundle on \(Y_j\). Note that on \(Y_j\) we have the \(T\)-equivariant short exact sequence
\[
    0\to \det(\cS_j)\otimes\C_{-t_1-\dots-t_j}\to \cO_{Y_j}\to \cO^{s_j}\to 0.
\]
Pulling back along \(p_j\) it gives the identity
\begin{equation}\label{eqn:div-bun}
    \det(\cS_j)=\C_{t_1+\dots+t_j}(1-\cO^{s_j})
\end{equation}
in \(\K_T(X)\).

\begin{cor}
    For any effective degree \(d\), {a $B$-stable Schubert variety $\Omega \subseteq X$}, and \(j=1,\dots, k\), we have
    \begin{equation}\label{eqn:2point-det}
        \egw{\det(\cS_j), \cO_\Omega}{d}=
        \begin{cases}
            0 & d_i>0\\
            \chi^T_X(\det(\cS_j)\cdot\cO_\Omega) & d_j=0
        .\end{cases}
    \end{equation}
\end{cor}

\begin{proof}
    By \eqn{div-bun}, we have 
\begin{equation*}
    \begin{aligned}
        \egw{\det(\cS_j), \cO_\Omega}{d}=\C_{t_1+\dots+t_j}\left(\egw{\cO_\Omega}{d}-\egw{\cO^{s_j},\cO_\Omega}{d}\right),
    \end{aligned}
\end{equation*}
and \eqn{2point-det} follows from \eqn{q=c1} and \eqn{q=c2divisor}. 
\end{proof}}

    Buch and Mihalcea have the following unpublished conjecture about 3-pointed (equivariant) Chevalley-type \(\K\)-theoretic Gromov--Witten invariants. 
    \begin{conj}[Buch--Mihalcea]\label{conj:full} 
        {For any effective degree \(d\) and $\sigma, \tau\in\K_T(X)$, we have}
        \begin{equation*}\label{eqn:conj}
                \egw{\det(\cS_j),\sigma,\tau}{d}=
                \begin{cases}
                    0 & d_j>0\\\chi_X^T(\det(\cS_j)\cdot\sigma\cdot\tau[d]) & d_j=0.
                \end{cases}
            \end{equation*}
        \end{conj}
        The conjecture has been proved in several special situations. For (cominuscule) Grassmannians,
        it follows from the ``quantum = classical'' results in \cite{buch.m:qk,chaput.perrin:rationality}, and for incidence varieties \(\Fl(1,n-1;n)\) it was recently proved in \cite{xu2021quantum}; see \Theorem{conj-incidence} below. Recently, \cite{sinha2024quantum} gave an  alternative proof and generalization of the first equation in the Grassmannian case. 

In the next proposition, we give a short proof that \Conjecture{full} implies part of \Cref{conj2:lambda_y} for all partial flag varieties. In \Section{fullflag}, we prove that \Conjecture{full} implies \Cref{conj2:lambda_y} when \(X\) is the complete flag variety \(\Fl(n)\). 

\begin{prop}\label{prop:spec_rel} Assuming that \Cref{conj:full} holds, the following relations hold in \(\QK_T(X)\):
\[
 \det(\cS_i)\star\det(\cS_{i+1}/\cS_i)=(1-q_i)\det(\cS_{i+1}),\quad i=1,\dots,k-1. 
 \]
 \end{prop}
\begin{proof} Let \(d=(d_1,\dots,d_k)\) be any degree and {with $i$ fixed, define the degree $d'$ by} 
\[
    {d_j'}=
    \begin{cases}
        d_j & j\neq i\\d_i-1 & j=i.
    \end{cases}
\] We need to show that for any $\sigma \in \K_T(X)$,
\begin{equation}\label{eqn:toprove} \egw  {\det(\cS_i),  \det(\cS_{i+1}/\cS_i), \sigma}{d} =  \egw{\det(\cS_{i+1}), \sigma}{d} - \egw{ \det(\cS_{i+1}), \sigma }{d'} \/, \end{equation}
{with the convention that the term involving $d'$ is omitted if $d_i=0$.}

If \(d_i\neq0\), then by \Cref{conj:full}, the left-hand side of \eqn{toprove} is equal to \(0\).  
{By \eqn{2point-det},
\[ \begin{split} 
    \egw{\det(\cS_{i+1}), \sigma}{d} &=
    \begin{cases}
        0 & d_{i+1}>0 \\
        \chi_X^T(\det(\cS_{i+1})\cdot\sigma) & d_{i+1}=0 
    \end{cases}\\ 
    &=\egw{ \det(\cS_{i+1}), \sigma }{d'}\ , 
\end{split}
\] 
}
which implies that the right-hand side of \eqn{toprove} is also equal to \(0\). 

If \(d_i=0\), then by \Conjecture{full} and \eqn{2point-det}, the left-hand side of \eqn{toprove} equals 
\[
    \chi_X^T\left(\det(\cS_i)\cdot\det(\cS_{i+1}/\cS_i)\cdot\sigma[d]\right)=\chi_X^T(\det(\cS_{i+1})\cdot\sigma[d])=\egw{\det(\cS_{i+1}),\sigma}{d},
\]
which equals the right-hand side of \eqn{toprove}.
\end{proof}

\section{The Whitney presentation for \texorpdfstring{\(\QK_T(\Fl(1,n-1;n))\)}{the equivariant quantum K ring of incidence varieties}}\label{sec:incidence}
In this section, we prove the Whitney presentation for 
$\QK_T(X)$, where $X=\Fl(1,n-1;n)$ for $n \ge 3$ is the incidence variety. {A key fact is that \Conjecture{full} has been proved in this case \cite{xu2021quantum}}. 
We start with a section introducing some preliminary results needed in the proof.

\subsection{Preliminaries} 
Recall that $X$ is a projective manifold of dimension $2n-3$, and that
\(0\subset\cS_1\subset\cS_2\subset\cS_3=\C^n\) is the flag of tautological bundles over $X$, where the \(\cS_1, \cS_2\) are of ranks $1, n-1$ respectively. Denote by $p_1: X \to \Gr(1,n)$ and $p_2: X \to \Gr(n-1,n)$ the natural projections.

It was proved in \cite{xu2021quantum}*{Corollary 4.6} that \Conjecture{full} holds for incidence varieties.

\begin{thm}[Xu]\label{thm:conj-incidence} Let $d$ be an effective degree. 
    For any Schubert variety $\Omega$, and \(k=1,\ 2\), 
    \begin{equation}\label{eqn:q=c3}
        \egw{\cO^{s_k},\sigma,\cO_\Omega}{d}=\begin{cases}
            \chi^T_X(\sigma\cdot\cO_{\Gamma_d(\Omega)}) & d_k>0\\
            \chi^T_X(\cO^{s_k}\cdot\sigma\cdot\cO_{\Gamma_d(\Omega)}) & d_k=0
        .\end{cases}
    \end{equation}
\end{thm}

\begin{cor}\label{cor:KGWinvarints}
    For any effective degree \(d\), and any Schubert variety $\Omega$,
    \begin{equation}\label{eqn:q=c31-O1}
        \egw{\cS_1, \sigma, \cO_\Omega}{d}=
        \begin{cases}
            0 & d_1>0\\\chi^T_X(\cS_1\cdot\sigma\cdot\cO_{\Gamma_d(\Omega)}) & d_1=0
        ,\end{cases}
    \end{equation}
    \begin{equation}\label{eqn:q=c31-O2}
        \egw{\det(\cS_2), \sigma, \cO_\Omega}{d}=\begin{cases}
            0 & d_2>0\\\chi^T_X(\det(\cS_2)\cdot\sigma\cdot\cO_{\Gamma_d(\Omega)}) & d_2=0
        .\end{cases}
    \end{equation}
\end{cor}

\begin{proof}
    Since \(\cS_1=\C_{t_1}(1-\cO^{s_1})\) by \eqn{div-bun}, we have 
\begin{equation*}
    \begin{aligned}
        \egw{\cS_1, \sigma, \cO_\Omega}{d}=\C_{t_1}\left(\egw{\sigma,\cO_\Omega}{d}-\egw{\cO^{s_1},\sigma,\cO_\Omega}{d}\right),
    \end{aligned}
\end{equation*}
and then \eqn{q=c31-O1} follows from \eqn{q=c2} and \eqn{q=c3}. The proof of \eqn{q=c31-O2} is similar.
\end{proof}

\begin{remark}\label{remark:divisor_poly}
    {Let \(\QK_T^\poly(X)\subseteq\QK_T(X)\) be the subring generated by \(\cO^{s_1}\) and \(\cO^{s_2}\) over the ground ring \(\K_T(\pt)[q_1,q_2]\).
    \cite{xu2021quantum}*{Algorithm 4.16} 
 gives an algorithm that recursively expresses any Schubert class as a polynomial in \(\cO^{s_1},\ \cO^{s_2}\) with coefficients in \(\K_T(\pt)[q_1,q_2]\). Combined with \cite{xu2021quantum}*{Theorem 4.5}, this means that when expressing the product of two Schubert classes as a linear combination of Schubert classes in \(\QK_T(X)\), the coefficients are always in \(\K_T(\pt)[q_1,q_2]\). Therefore, 
\(\QK_T^\poly(X)\) can be identified with \(\K_T(X)\otimes\Z[q_1,q_2]\) as a module. Because of \eqn{div-bun}, 
\(\QK_T^\poly(X)\) is also generated by \(\cS_1,\ \det(\cS_2)\) over \(\K_T(\pt)[q_1,q_2]\). }
\end{remark}

For convenience, we restate the following curve neighborhood computations from \cite[\S 2.2.2]{xu2021quantum}.   
\begin{lemma} {The curve neighborhoods of a Schubert variety $\Omega \subseteq X$ are:}
    \begin{equation}\label{eqn:curvenbhd}
        \Gamma_d(\Omega)=\begin{cases}
            p_1^{-1}\left(p_1(\Omega)\right)& d_1=0,\ d_2>0\\
            p_2^{-1}\left(p_2(\Omega)\right)&d_1>0,\ d_2=0\\
            X& d_1>0,\ d_2>0.
        \end{cases}
    \end{equation}
\end{lemma}

\subsection{Quantum \(\K\) Whitney relations for incidence varieties}
We prove in \Cref{thm:rels} of this section two equivalent quantized versions of the Whitney relations for the 
quantum K ring of incidence varieties. 

{
We shall use the classical Whitney relations in \(\K_T(X)\):
\begin{equation}\label{eqn:kincidence}
    \lambda_y(\cS_{j})\cdot\lambda_y(\cS_{{j+1}}/\cS_{j})=\lambda_y(\cS_{{j+1}}),\quad j=1,2.
\end{equation}
Equivalently, 
\begin{equation}\label{eqn:w1}
    \wedge^{\ell}(\cS_2/\cS_1)+\cS_1\cdot\wedge^{\ell-1}(\cS_2/\cS_1)=\wedge^\ell(\cS_2),\quad \ell=1,\dots,n-1
\end{equation}
and 
\begin{equation}\label{eqn:w2}
    \wedge^\ell\cS_2+\wedge^{\ell-1}\cS_2\cdot\C^n/\cS_2=\wedge^\ell\C^n, \quad\ell=1,\dots,n.
\end{equation}

We also regard the incidence variety $X$ as a
Grassmann bundle in two ways, via the equivariant projections
\[
    p_1: X = \mathbb{G}(n-2, \C^n/\cS_1) \to Y_1=\Gr(1,n)
\]
and 
\[
    p_2: X = \mathbb{G}(1, \cS_2)\to Y_2=\Gr(n-1,n).
\] }

The next proposition will turn out to be a restatement
of the second relation in \Theorem{rels}. 

\begin{prop}\label{prop:rels1} For any $1 \le \ell \le n$, the
    following relation holds in \(\QK_T(X)\):
    \begin{equation}\label{eqn:3cases}
            \det(\cS_2)\star\left(\wedge^\ell\C^n-\wedge^\ell \cS_2\right)=\wedge^n\C^n \left(\wedge^{\ell -1}\cS_2-q_2\wedge^{\ell -1}\cS_1\right).
    \end{equation}
\end{prop}
\begin{proof}
    To prove \eqn{3cases}, it suffices to prove that
    \begin{equation}\label{eqn:GW}
        \egw{\det(\cS_2),(\wedge^\ell \C^n-\wedge^\ell  \cS_2),\cO_\Omega}{d}=\wedge^n\C^n\left(\egw{\wedge^{\ell -1}\cS_2,\cO_\Omega}{d}-\egw{\wedge^{\ell -1}\cS_1,\cO_\Omega}{d''}\right)
    \end{equation}
    for any Schubert variety \(\Omega\subseteq X\) and effective degree \(d=(d_1,d_2)\), where \(d''\coloneqq(d_1,d_2-1)\). This follows from \eqn{def} and the fact that Schubert classes form a basis for \(\K_T(X)\) over \(\K_T(\pt)\).

First assume \(d_2=0\). 
Then by \eqn{q=c31-O2}, the left-hand side of \eqn{GW} is equal to
    \begin{equation*}
        \begin{aligned}
            \chi^T_X\left(\det(\cS_2)\cdot(\wedge^\ell\C^n-\wedge^\ell \cS_2)\cdot\cO_{\Gamma_d(\Omega)}\right)&=\chi^T_X\left(\det(\cS_2)\cdot(\C^n/\cS_2)\cdot\wedge^{\ell-1}\cS_2\cdot\cO_{\Gamma_d(\Omega)}\right)\\
            &=\chi_X^T\left(\wedge^n\C^n\cdot\wedge^{\ell-1}\cS_2\cdot\cO_{\Gamma_d(\Omega)}\right)\\
            &=\wedge^n\C^n\cdot \chi^T_X\left(\wedge^{\ell-1}\cS_2\cdot\cO_{\Gamma_d(\Omega)}\right),
        \end{aligned}
    \end{equation*}
where the first equality utilizes \eqn{w2}, and the second follows again from \eqn{w2} in the special case 
$\ell =n$. By \eqn{q=c2}, this equals the right-hand side of \eqn{GW}.    

Now assume \(d_2>0\). By \eqn{q=c31-O2}, the left-hand side of \eqn{GW} is equal to \(0\). 
It suffices to show that for any $0 \le k \le n-1$, 
\begin{equation}\label{eqn:right-hand side}
\egw{\wedge^k\cS_2,\cO_\Omega}{d}=\egw{\wedge^k\cS_1,\cO_\Omega}{{d''}}.
\end{equation}
By \eqn{q=c2}, equation \eqn{right-hand side} is equivalent to 
    \begin{equation*}
        \chi_X^T\left(\wedge^k\cS_2\cdot\cO_{\Gamma_d(\Omega)}\right)=\chi_X^T\left(\wedge^k\cS_1\cdot\cO_{\Gamma_{d''}(\Omega)}\right).
    \end{equation*}
When \(k=0\) both sides are equal to $1$, and we assume from now on that \(k>0\). By \eqn{q=c2} and the projection formula,
    \begin{equation*}
        \begin{aligned}
            \chi^T_X\left(\wedge^k\cS_2\cdot\cO_{\Gamma_d(\Omega)}\right)=\chi^T_{Y_2}\left(\wedge^k\cS_2\cdot{p_2}_*\cO_{\Gamma_d(\Omega)}\right)=\chi^T_{Y_2}\left(\wedge^k\cS_2\cdot\cO_{p_2\left(\Gamma_d(\Omega)\right)}\right).
        \end{aligned}
    \end{equation*}
    When \(d_1>0\), we have \(\Gamma_d(\Omega)=X\) by \eqn{curvenbhd} and \(p_2\left(\Gamma_d(\Omega)\right)=Y_2\). Therefore, 
    \begin{equation*}
        \chi^T_{Y_2}\left(\wedge^k\cS_2\cdot\cO_{p_2\left(\Gamma_d(\Omega)\right)}\right)=\chi^T_{Y_2}\left(\wedge^k\cS_2\right)=0
    \end{equation*} 
    by \Cref{prop:relativeBWB}.
    By \eqn{2point-det}, we have
    \(
        \chi_X^T\left(\wedge^k\cS_1\cdot\cO_{\Gamma_{d''}(\Omega)}\right)=0,
    \)
    proving \eqn{right-hand side} in this case. When \(d_1=0\), by \eqn{curvenbhd}, \(\Gamma_d(\Omega)=p_1^{-1}\left(p_1(\Omega)\right)\), and therefore 
    \begin{equation*}
        p_1(\Omega)\supseteq p_1\left(\Gamma_d(\Omega)\right)\supseteq p_1\left(\Gamma_{d''}(\Omega)\right)\supseteq p_1(\Omega), 
    \end{equation*}
    forcing all of them to be equal. By the projection formula
   and \eqn{w1},
    \begin{equation*}
        \begin{aligned}
            \chi^T_X\left(\wedge^k\cS_2\cdot\cO_{\Gamma_d(\Omega)}\right)=&\chi^T_X\left(\wedge^k\cS_2\cdot p_1^*\cO_{p_1(\Omega)}\right)=\chi_{Y_1}^T\left({p_1}_*(\wedge^k\cS_2)\cdot\cO_{p_1(\Omega)}\right)\\=&\chi_{Y_1}^T\left({p_1}_*\left(\cS_1\cdot\wedge^{k-1}(\cS_2/\cS_1)+\wedge^k(\cS_2/\cS_1)\right)\cdot\cO_{p_1(\Omega)}\right)\\=&\chi_{Y_1}^T\left(\left(\cS_1\cdot{p_1}_*\left(\wedge^{k-1}(\cS_2/\cS_1)\right)+{p_1}_*\left(\wedge^k(\cS_2/\cS_1)\right)\right)\cdot\cO_{p_1(\Omega)}\right)\\=&
            \begin{cases}
                \chi_{Y_1}^T\left(\cS_1\cdot\cO_{p_1(\Omega)}\right) & k=1\\
                0 & k>1,
            \end{cases}
        \end{aligned}
    \end{equation*}
 where the last equality follows from \Cref{prop:relativeBWB} because $\cS_2/\cS_1$ is the tautological
 subbundle of the Grassmann bundle $X=\mathbb{G}(n-2, \C^n/\cS_1) \to Y_1$. The claim follows from combining this
 and    
 \begin{equation*}
 \begin{split}
        \chi_X^T\left(\wedge^k\cS_1\cdot\cO_{\Gamma_{d''}(\Omega)}\right)& =\chi_{Y_1}^T\left(\wedge^k\cS_1\cdot\cO_{p_1(\Gamma_{d''}(\Omega))}\right)\\ & =\chi_{Y_1}^T\left(\wedge^k\cS_1\cdot\cO_{p_1(\Omega)}\right)\\ & =\begin{cases}
            \chi_{Y_1}^T\left(\cS_1\cdot\cO_{p_1(\Omega)}\right) & k=1\\
            0 & k>1.
        \end{cases}
  \end{split}
    \end{equation*}
\end{proof}

{The following theorem ``quantizes'' the classical \(\K\)-theoretic Whitney relations \eqn{kincidence}. }
\begin{thm}\label{thm:rels}
        The following relations hold in \(\QK_T(X)\):
        \begin{equation}\label{eqn:rel1}
            \lambda_y(\cS_1)\star\lambda_y(\cS_2/\cS_1)=\lambda_y(\cS_2)-q_1 y^{n-1}\det(\cS_2),
        \end{equation}
        \begin{equation}\label{eqn:rel2}
                \la_y(\cS_2)\star\la_y(\C^n/\cS_2)=\la_y(\C^n)-q_2[\la_y(\C^n)-\la_y(\cS_2)-(\la_y(\C^n/\cS_2)-1)\star \la_y(\cS_1)].
        \end{equation}
{As a consequence, \Cref{conj2:lambda_y} holds for incidence varieties.}        
    \end{thm}
    
    \begin{proof}
        The proof of \eqn{rel1} is similar to that of \Cref{prop:rels1},
        but we provide the details for completeness. To prove \eqn{rel1}, it suffices to prove that 
        \begin{equation}\label{eqn:GW1}
            \egw{\la_y(\cS_1),\la_y(\cS_2/\cS_1),\cO_\Omega}{d}=\egw{\la_y(\cS_2),\cO_\Omega}{d}-y^{n-1}\egw{\det(\cS_2),\cO_\Omega}{d'}
        \end{equation}
for any Schubert variety \(\Omega\subseteq X\) and effective degree \(d=(d_1,d_2)\),
where \(d'\coloneqq (d_1-1,d_2)\). 
    
When \(d_1=0\), equation \eqn{GW1} follows from \eqn{q=c2} and \eqn{q=c31-O1}, because 
        \[
            \chi_X^T\left(\la_y(\cS_1)\cdot\la_y(\cS_2/\cS_1)\cdot\cO_{\Gamma_d(\Omega)}\right)=\chi_X^T\left(\la_y(\cS_2)\cdot\cO_{\Gamma_d(\Omega)}\right)\/
        \] 
        by \eqn{kincidence}. 
        Now assume that \(d_1>0\). By \eqn{q=c2} and \eqn{q=c31-O1}, the left-hand side of \eqn{GW1} is equal to
        \begin{equation*}
            \chi_X^T\left(\la_y(\cS_2/\cS_1)\cdot\cO_{\Gamma_d(\Omega)}\right),
        \end{equation*}
        and the right-hand side of \eqn{GW1} is equal to
        \begin{equation*}
            \chi_X^T\left(\la_y(\cS_2)\cdot\cO_{\Gamma_d(\Omega)}\right)-y^{n-1}\chi_X^T\left(\det(\cS_2)\cdot\cO_{\Gamma_{d'}(\Omega)}\right).
        \end{equation*}
Note that \(\Gamma_d(\Omega)=p_2^{-1}(p_2\left(\Gamma_d(\Omega)\right))\) by \eqn{curvenbhd}.
By the projection formula, 
        \begin{equation*}
            \begin{aligned}
                \chi_X^T\left(\la_y(\cS_2/\cS_1)\cdot\cO_{\Gamma_d(\Omega)}\right)&=\chi_X^T\left(\la_y(\cS_2/\cS_1)\cdot 
                p_2^*\cO_{p_2\left(\Gamma_d(\Omega)\right)}\right)\\
                &=\chi_{Y_2}^T\left({p_2}_*\left(\la_y(\cS_2/\cS_1)\right)\cdot\cO_{p_2\left(\Gamma_d(\Omega)\right)}\right)\\
                &=\chi_{Y_2}^T\left(\la_y(\cS_2)_{\leq n-2}\cdot\cO_{p_2\left(\Gamma_d(\Omega)\right)}\right)\\
                &=\chi_{Y_2}^T\left(\la_y(\cS_2)_{\leq n-2}\cdot{p_2}_*\cO_{\Gamma_d(\Omega)}\right)\\
                &=\chi_X^T\left(\la_y(\cS_2)_{\leq n-2}\cdot\cO_{\Gamma_d(\Omega)}\right),
            \end{aligned} 
        \end{equation*}
 where 
 \[
    \la_y(\cS_2)_{\leq n-2}=1+y\cS_2+\dots+y^{n-2}\wedge^{n-2}\cS_2
\]
and the third equality follows from \Cref{prop:relativeBWB}.
Therefore, it suffices to show that 
        \begin{equation}\label{eqn:d1>0}
            \chi_X^T\left(\det(\cS_2)\cdot\cO_{\Gamma_d(\Omega)}\right)=\chi_X^T\left(\det(\cS_2)\cdot\cO_{\Gamma_{d'}(\Omega)}\right).
        \end{equation} 
 When \(d_2=0\), by \eqn{curvenbhd}, we have \(\Gamma_d(\Omega)=p_2^{-1}(p_2(\Omega))\), and as in the proof of 
 \Cref{prop:rels1}, we have    
        \begin{equation*}
            p_2(\Omega)\supseteq p_2\left(\Gamma_d(\Omega)\right)\supseteq p_2\left(\Gamma_{d'}(\Omega)\right)\supseteq p_2(\Omega), 
        \end{equation*}
        forcing all of them to be equal. 
        By the projection formula, 
        \begin{equation*}
            \begin{aligned}
                \chi_X^T\left(\det(\cS_2)\cdot\cO_{\Gamma_d(\Omega)}\right)=&\chi_{Y_2}^T\left(\det(\cS_2)\cdot{p_2}_*\cO_{\Gamma_d(\Omega)}\right)\\
                =&\chi_{Y_2}^T\left(\det(\cS_2)\cdot\cO_{{p_2}(\Gamma_d(\Omega))}\right).
            \end{aligned}
        \end{equation*} 
        Similarly,
        \begin{equation*}
            \chi_X^T\left(\det(\cS_2)\cdot\cO_{\Gamma_{d'}(\Omega)}\right)=\chi_{Y_2}^T\left(\det(\cS_2)\cdot\cO_{{p_2}(\Gamma_{d'}(\Omega))}\right).
        \end{equation*}
Equation \eqn{d1>0} follows because \({p_2}(\Gamma_d(\Omega))={p_2}(\Gamma_{d'}(\Omega))\). When \(d_2>0\), by \eqn{curvenbhd}, \(\Gamma_d(\Omega)=X\) and \(\Gamma_{d'}(\Omega)=p_1^{-1}(p_1(\Gamma_{d'}(\Omega)))\). Therefore, by the projection formula and \Cref{prop:relativeBWB},
        \begin{equation*}
\chi_X^T\left(\det(\cS_2)\cdot\cO_{\Gamma_d(\Omega)}\right)=\chi_X^T\left(\det(\cS_2)\cdot\cO_X\right)=\chi_X^T\left(\det(\cS_2)\right)=\chi_{Y_2}^T\left(\det(\cS_2)\right)=0,
        \end{equation*}
and
        \begin{equation*}
            \begin{aligned}
                \chi_X^T\left(\det(\cS_2)\cdot\cO_{\Gamma_{d'}(\Omega)}\right)=&\chi_X^T\left(\wedge^{n-2}(\cS_2/\cS_1)\cdot \cS_1\cdot\cO_{\Gamma_{d'}(\Omega)}\right)\\
                =&\chi_X^T\left(\wedge^{n-2}(\cS_2/\cS_1)\cdot p_1^*\left(\cS_1\cdot\cO_{p_1(\Gamma_{d'}(\Omega))}\right)\right)\\
                =&\chi_{Y_1}^T\left({p_1}_*\left(\wedge^{n-2}(\cS_2/\cS_1)\right)\cdot \cS_1\cdot\cO_{p_1(\Gamma_{d'}(\Omega))}\right)=0,
            \end{aligned}
        \end{equation*}
from which \eqn{d1>0} follows.

        We derive \eqn{rel2} from \eqn{3cases}. First, note that
        \begin{equation}\label{eqn:sum}
                \la_y(\cS_2)\star\la_y(\C^n/\cS_2)=\sum_{\ell=0}^n y^\ell(\wedge^\ell \cS_2+\wedge^{\ell-1}\cS_2\star(\C^n/\cS_2)).
        \end{equation} 
 Equation \eqn{3cases} applied to $\ell=1$ gives that 
 $(\C^n/\cS_2)\star\det(\cS_2)={\wedge^n \C^n}(1-q_2)$.
 Then by associativity 
        \begin{equation*}
                \begin{aligned}
                        \frac{(\C^n/\cS_2)\star\det(\cS_2)\star(\wedge^\ell\C^n-\wedge^\ell \cS_2)}{\wedge^n\C^n}&=(1-q_2)(\wedge^\ell\C^n-\wedge^\ell \cS_2)\\
                        &=(\C^n/\cS_2)\star(\wedge^{\ell-1}\cS_2-q_2\wedge^{\ell-1}\cS_1).
                \end{aligned}      
        \end{equation*}
After rearranging the terms in the last equality, we obtain 
        \begin{equation}\label{eqn:subs}
                \begin{aligned}
                        \wedge^\ell \cS_2+\wedge^{\ell-1}\cS_2\star(\C^n/\cS_2)=\wedge^\ell\C^n-q_2\left(\wedge^\ell\C^n-\wedge^\ell \cS_2-(\C^n/\cS_2)\star\wedge^{\ell-1}\cS_1\right).
                \end{aligned}
        \end{equation}
        Plugging \eqn{subs} into \eqn{sum}, we obtain \eqn{rel2}:
        \begin{equation*}
                \la_y(\cS_2)\star\la_y(\C^n/\cS_2)=\la_y(\C^n)-q_2[\la_y(\C^n)-\la_y(\cS_2)-(\la_y(\C^n/\cS_2)-1)\star \la_y(\cS_1)].
        \end{equation*}
Finally, the proof ends by observing that the relations just proved are equivalent to the relations from
\Cref{conj2:lambda_y}. Indeed, the equivalence of \eqn{rel1} and the first relation (\(j=1\))
from \Cref{conj2:lambda_y} uses that 
\[ \det \cS_1 \star \det \cS_2/\cS_1 = (1-q_1) \det \cS_2 \/, \]
and is a special case of both \eqn{rel1} and
\Cref{conj2:lambda_y}.
The equivalence of \eqn{rel2} and the second relation (\(j=2\))
from \Cref{conj2:lambda_y} follows by multiplying by $1-q_2$ and rearranging terms.
This finishes the proof.
\end{proof}

\subsection{The QK Whitney presentation}\label{sec:presentation}
The goal of this section 
is to prove, in \Cref{thm:incidence} and \Cref{cor:poly} respectively,
the Whitney presentation for $\QK_T(X)$ and a ``localized'' version, 
which holds over a subring of the power series ring. 

We start by recalling the Whitney relations in the case of incidence varieties. Recall that
\[ S = \K_T(\pt)[X^{(1)}_1,e_1(X^{(2)}),\dots,e_{{n-1}}(X^{(2)}),e_1(Y^{(1)}),\dots,e_{n-2}(Y^{(1)}),Y^{(2)}_1] \/,\] 
and let $I_q \subset S[\![q]\!]$ be the ideal generated by the coefficients of \(y\) in
\begin{equation}\label{eqn:r-1}
    (1+yX^{(1)}_1)\prod_{\ell=1}^{n-2}(1+yY^{(1)}_\ell)-\prod_{\ell=1}^{n-1}(1+yX^{(2)}_\ell)+y^{n-1}\frac{q_1}{1-q_1}(\prod_{\ell=1}^{n-2}Y^{(1)}_\ell)X^{(1)}_1
\end{equation}
and 
\begin{equation}\label{eqn:r-2}
   \left(\prod_{\ell=1}^{n-1}(1+yX^{(2)}_\ell)\right)(1+yY^{(2)}_1)-\prod_{\ell=1}^{n}(1+yT_\ell)+y\frac{q_2}{1-q_2}Y^{(2)}_1\left(\prod_{\ell=1}^{n-1}(1+yX^{(2)}_\ell)-(1+yX^{(1)}_1)\right).
\end{equation}

\begin{thm}\label{thm:incidence}
    There is an isomorphism of \(\K_T(\pt)[\![q_1,q_2]\!]\)-algebras 
    {\[
         \Psi:S[\![q]\!]/I_q\to \QK_T(X)
     \]}
     sending \(X^{(1)}_1\) to \(\cS_1\), \(e_\ell(Y^{(1)})\) to \(\wedge^\ell(\cS_2/\cS_1)\), \(e_\ell(X^{(2)})\) to \(\wedge^\ell(\cS_2)\), and \(Y^{(2)}_1\) to \(\C^n/\cS_2\).
 \end{thm}
 \begin{proof}
    This follows from \Theorem{all_relations} and \Theorem{rels}.
 \end{proof}

 Define the following polynomial and localized versions of submodules of \(\QK_T(X)\): 
\[ \QK_T^\poly(X)\coloneqq\K_T(X) \otimes_{\K_T(\pt)}\K_T(\pt)[q_1,q_2] \subset \QK_T^{\rm{loc}}(X)\coloneqq \K_T(X) \otimes_{\K_T(\pt)}\K_T(\pt)[q_1,q_2]_{1+\langle q_1,q_2\rangle} 
\/.\]
Since the product of two Schubert classes in $\QK_T(X)$ involves only coefficients in 
$\K_T(\pt)[q_1, q_2]$ (cf.~\Cref{remark:divisor_poly}, see also \cite{anderson2022finite} and \cite{kato:loop}*{Cor.~4.16}), 
it follows that both
$\QK_T^\poly(X)$ and $\QK_T^{\rm{loc}}(X)$ are in fact subalgebras of $\QK_T(X)$ over the 
appropriate ground rings.

Consider the polynomial ring $S[q]$ and let $I_q^{\poly}\subset S[q]$
be the ideal generated by the coefficients of \(y\) in 
\begin{equation*}\label{eqn:r1}
    (1+y X^{(1)}_1)\prod_{\ell=1}^{n-2}(1+yY^{(1)}_\ell)-\prod_{\ell=1}^{n-1}(1+yX^{(2)}_\ell)+q_1y^{n-1}\prod_{\ell=1}^{n-1}X^{(2)}_\ell
\end{equation*}
and 
\begin{multline*}\label{eqn:r2}
    \left(\prod_{\ell=1}^{n-1}(1+yX^{(2)}_\ell)\right)(1+yY^{(2)}_1)-\prod_{\ell=1}^{n}(1+yT_\ell)\\+q_2\left(\prod_{\ell=1}^{n}(1+yT_\ell)-\prod_{\ell=1}^{n-1}(1+yX^{(2)}_\ell)-yY^{(2)}_1(1+yX^{(1)}_1)\right).
\end{multline*}
These relations formalize the relations from \Cref{thm:rels}.
By \cite[Prop.~10.13]{AM:intro}, the completion of the $S[q]$-module $S[q]/I_q^\poly$ 
along the ideal $\langle q\rangle\coloneqq\langle q_1, q_2 \rangle$ is
\[ S[q]/I_q^\poly \otimes_{S[q]} S[\![q]\!] = S[\![q]\!]/I_q \/, \]
where we have used that $I_q^\poly S[\![q]\!] = I_q$, as shown in the end of the proof of \Theorem{rels}.

Set $S[q]_\loc:=S[q]_{1+\langle q\rangle}$, and define the ideal
$I_q^\loc \coloneqq I_q^\poly S[q]_\loc$
generated by \eqn{r-1} and \eqn{r-2} in the localized ring 
$S[q]_\loc$. Since
\[ S[q]_\loc/I_q^\loc = (S[q]/I_q^\poly)_{1+\langle q\rangle} \/, \]
it follows from \cite[p.~110]{AM:intro} that $S[q]_\loc/I_q^\loc$ is a subring
of $S[\![q]\!]/I_q$.
\begin{cor}\label{cor:poly}
    The isomorphism \(\Psi\) restricts to an isomorphism 
    of \(S[q]_\loc\)-algebras 
    \begin{equation*}
        \Psi_{\rm{loc}}: S[q]_\loc/I_q^\loc\to \QK_T^{\rm{loc}}(X).
    \end{equation*} 
\end{cor}

\begin{proof} 
    The injectivity follows from the injetivity of \(\Psi\). The surjectivity
follows from \Cref{remark:divisor_poly}, which implies that $\QK_T^{\rm{loc}}(X)$ is generated over
$S[q]_\loc$
by the (line) bundles $\cS_1$ and $\det \cS_2$. 
\end{proof} 

\begin{remark}\label{rmk:poly} The attentive reader may have noticed that, 
    by \Cref{thm:rels}, there is also a natural ring homomorphism
    \[ \Psi_\poly: S[q]/I_q^\poly \to \QK_T^\poly(X) \]
    defined by the same formula as $\Psi$. Furthermore, \Cref{remark:divisor_poly}
    implies that $\Psi_\poly$ is surjective. {However, $\Psi_\poly$ is not injective, see \Example{fl3} below. 
    The kernel agrees with that of the localization map \(S[q]/I_q^\poly\to S[q]_\loc/I_q^\loc\) because of the commutative diagram 
    \[
        \begin{CD}S[q]/I_q^\poly @>\Psi_\poly>> \QK_T^\poly(X)\\ @VVV @VVV\\S[q]_\loc/I_q^\loc @>\Psi_{\rm{loc}}>>\QK_T^{\rm{loc}}(X)
        \end{CD}
    \]
    and the injectivity of \(\QK_T^\poly(X)\to\QK_T^{\rm{loc}}(X)\). }
\end{remark}

{
   \begin{example}\label{example:fl3}
        Let \(X=\Fl(1,2;3)\), then 
        \[I_q^\poly\subset S[q]=\K_T(\pt)[X_1^{(1)},e_1(X^{(2)}),e_2(X^{(2)}),Y^{(1)}_1,Y^{(2)}_1][q]\] is the ideal generated by the relations
        \begin{equation*}
            X_1^{(1)}+Y_1^{(1)}-e_1(X^{(2)}),
        \end{equation*}
        \begin{equation*}
            X_1^{(1)}Y_1^{(1)}-(1-q_1)e_2(X^{(2)}),
        \end{equation*}
        \begin{equation*}
            (1-q_2)\left(e_1(X^{(2)})+Y_1^{(2)}-e_1(T)\right),
        \end{equation*}
        \begin{equation*}
            \left(e_1(X^{(2)})-q_2X_1^{(1)}\right)Y_1^{(2)}-(1-q_2)\left(e_2(T)-e_2(X^{(2)})\right),
        \end{equation*}
        \begin{equation*}
            e_2(X^{(2)})Y_1^{(2)}-(1-q_2)e_3(T).
        \end{equation*}
        {Focusing on the third relation}, one easily checks that \(e_1(X^{(2)})+Y_1^{(2)}-e_1(T)\) is a nonzero element in the kernel of both \(\Psi_\poly\) and the localization map \(S[q]/I_q^\poly\to S[q]_\loc/I_q^\loc\).
    \end{example}
    }

\subsection{A physics derivation of the QK Whitney presentation}\label{sec:physics}
{Continuing the circle of ideas first used for Grassmannians in \cite{Gu:2020zpg,gu2022quantum}, we derived from physics in \cite{gu2023quantum} a ``{Coulomb branch presentation}''
of the quantum K ring of any partial flag variety. This presentation was obtained as the
critical locus of a certain \emph{one-loop twisted superpotential} $\mathcal{W}$
which arises in the study of $3$d gauged linear sigma models (GLSM). In this section, we recall the derivation of the Coulomb branch presentation for incidence varieties and show that it is essentially the same as the Whitney presentation.  The same ideas apply to any partial flag variety; we review this briefly in \Remark{general}.}

As before, consider the variables $X^{(1)}_1, X^{(2)}_1, \ldots, X^{(2)}_{n-1}$.
The twisted superpotential $\mathcal{W}$ is defined as
\begin{eqnarray*}
\mathcal{W} & = &
\frac{1}{2} (n-2) \sum_{a=1}^{n-1} \left( \ln X^{(2)}_a \right)^2
\: - \:
\sum_{1 \le a < b \le n-1} \left(\ln X^{(2)}_{a}\right) 
\left(\ln X^{(2)}_{b}\right)
\nonumber \\
& &
\: + \:
\left( \ln q_1 \right) \left( \ln X^{(1)}_1 \right)
\: + \:
\left( \ln \left( (-1)^{n-2} q_2 \right) \right) \sum_{a=1}^{n-1}
\left( \ln X^{(2)}_a \right)
\nonumber \\
& & 
\: + \: 
\sum_{a=1}^{n-1} {\rm Li}_2 \left( X^{(1)}_1 / X^{(2)}_a \right)
\: + \:
\sum_{a=1}^{n-1} \sum_{i=1}^n {\rm Li}_2 \left( X^{(2)}_a / T_i \right)
\end{eqnarray*}
on the subset
of the torus $X^{(1)}_1 \cdot X^{(2)}_1 \cdot \ldots \cdot X^{(2)}_{n-1} \neq 0$
{where $X^{(2)}_i \neq X^{(2)}_j$ for $i \neq j$}. Here, $\mathrm{Li}_2$ is the dilogarithm function, satisfying
$x \frac{\partial}{\partial x} \mathrm{Li}_2(x) = - \ln (1-x)$. 

As a side remark, the more general superpotential associated to any GLSM also depends on
certain Chern-Simons levels. Here we have already chosen the levels giving the 
quantum K theory ring. 

The (unsymmetrized) Coulomb branch relations are given by 
\[ \mathrm{exp} \left(\frac{\partial \mathcal{W}}{\partial \ln (X^{(1)}_1)}\right) = 1 \/; \quad 
\mathrm{exp} \left(\frac{\partial \mathcal{W}}{\partial \ln (X^{(2)}_i)}\right) = 1\/, \quad 1 \le i \le n-1 \/. \]
Calculating derivatives one obtains 

\begin{equation}\label{eqn:deriv1}
q_1 \: = \:  \prod_{i=1}^{n-1} \left(
1 - \frac{ X_1^{(1)} }{ X_i^{(2)} } \right),
\end{equation}
\begin{equation} \label{eqn:deriv2}
(-1)^{n-2} q_2 
\left( 1- \frac{X_1^{(1)}}{X_k^{(2)}} \right)
\: = \:
\left( \prod_{i=1}^{n-1} \frac{X_i^{(2)}}{X_k^{(2)}} \right)
\prod_{j = 1}^n \left( 1 - \frac{ X_k^{(2)} }{  T_{j} } \right),\ k= 1, \dots, n-1.
\end{equation}

Equations \eqn{deriv1} and \eqn{deriv2} are a special case of the Bethe Ansatz equations \cite{koroteev}*{Equation (27)}
and the gauge/Bethe correspondence of Nekrasov and Shatashvili \cite{nekrasov.shatashvili}.
Investigating these connections is of high interest, but beyond the scope of the current paper.

We observe that \(\xi=X_1^{(1)}\) is a solution to the equation 
\begin{equation*}
    X_1^{(1)}\sum_{l=0}^{n-1}(-1)^\ell e_\ell(X^{(2)})\xi^{n-1-\ell}+(-1)^{n-2}q_1 e_{n-1}(X^{(2)})\xi=0,
\end{equation*}
and we denote the remaining \(n-2\) solutions by \(\overline{X}^{(1)}=(\overline{X}_1^{(1)},\dots,\overline{X}_{n-2}^{(1)})\); similarly, we have that \(\xi=X^{(2)}_1,\dots,X^{(2)}_{n-1}\) are solutions to 
\begin{equation*}
    e_{n-1}(X^{(2)})\sum_{\ell=0}^{n}(-1)^\ell e_\ell(T)\xi^{n-\ell}+q_2 e_n(T)\left(-\xi^{n-1}+X_1^{(1)}\xi^{n-2}\right)=0,
\end{equation*}
and we denote the remaining solution by \(\overline{X}^{(2)}_1\). Then, using Vieta's formulae on this expression regarded as a polynomial in $\xi$, we deduce equations 
\begin{equation*}  
e_\ell(\overline{X}^{(1)})+X_1^{(1)} e_{\ell-1}(\overline{X}^{(1)})-e_{\ell}( X^{(2)} ) =\begin{cases} 0 & \ell = 1, \dots, n-3, n-1 \\
q_1 e_{n-1}( X^{(2)} )/X^{(1)}_1  &  \ell = n-2 \/,
\end{cases}
\end{equation*}
and
\begin{equation*}  
e_\ell(X^{(2)})+e_{\ell-1}(X^{(2)})\overline{X}^{(2)}_1-
e_{\ell}(T) = 
\begin{cases}
q_2 e_{n}(T)/e_{n-1}(X^{(2)}) & \ell=1\\
q_2 e_n(T)X^{(1)}_1/e_{n-1}(X^{(2)}) & \ell=2\\
0 & \ell=3,\dots,n \/,    
\end{cases}
\end{equation*}
which simplify to the equations
\begin{equation}  \label{eqn:inc:final:1}
    e_\ell(\overline{X}^{(1)})+X_1^{(1)} e_{\ell-1}(\overline{X}^{(1)})-e_{\ell}( X^{(2)} ) = \: 
    \left\{ \begin{array}{cl}
    0 & \ell = 1, \dots, n-3, n-1\\q_1 e_{n-2}( \overline{X}^{(1)} )  &  \ell = n-2
    \end{array} \right.
    \end{equation}
    and
    \begin{equation}  \label{eqn:inc:final:2}
    e_\ell(X^{(2)})+e_{\ell-1}(X^{(2)})\overline{X}^{(2)}_1-
    e_{\ell}(T) = 
    \begin{cases}
    q_2 \overline{X}_1^{(2)} & \ell=1\\
    q_2 X^{(1)}_1 \overline{X}_1^{(2)} & \ell=2\\
    0 & \ell=3,\dots,n \/.
    \end{cases}
\end{equation}
Let 
\[
    \overline{\QK}_T(X)\coloneqq \K_T(\pt)[\![q_1,q_2]\!][X^{(1)}_1, e_i(X^{(2)}), e_j(\overline{X}^{(1)}),\overline{X}^{(2)}_1; 1 \le i \le n-1,\ 1 \le j \le n-2]/J_q
\]
be the ``Coulomb branch'' ring, where \(J_q\) is the ideal generated by the relations given by \eqn{inc:final:1} and \eqn{inc:final:2}.
\begin{prop}\label{prop:iso-Coulomb-Whitney}
    There is an isomorphism of \(\K_T(\pt)[\![q_1,q_2]\!]\)-algebras \(\Phi: \overline{\QK}_T(X) \to \QK_T(X)\) given by 
    \begin{equation}\label{eqn:Phi1}
        X^{(1)}_1\mapsto\cS_1;\quad e_k(X^{(2)})\mapsto\wedge^k(\cS_2),\ 1\leq k\leq n-1;
    \end{equation}
    and 
    \begin{equation}\label{eqn:Phi2}
        e_\ell(\overline{X}^{(1)})\mapsto
        \begin{cases}
            \wedge^{\ell}( \cS_2/\cS_1 ) & 1\leq\ell < n-2\\
            \det(\cS_2/\cS_1)/(1-q_1) & \ell=n-2 
        \end{cases};
        \quad \overline{X}^{(2)}_1\mapsto({\mathbb C}^n/\cS_2)/(1-q_2).
    \end{equation}
\end{prop}
\begin{proof}
    Let the morphism of \(\K_T(\pt)[\![q_1,q_2]\!]\)-algebras 
    \[
        \widetilde{\Phi}: \K_T(\pt)[\![q_1,q_2]\!][X^{(1)}_1, e_i(X^{(2)}), e_j(\overline{X}^{(1)}),\overline{X}^{(2)}_1; 1 \le i \le n-1,\ 1 \le j \le n-2]\to\QK_T(X)
    \]
    be defined by \eqn{Phi1} and \eqn{Phi2}. Note that \(\widetilde{\Phi}\) is surjective. The morphism \(\Phi\) sends equations \eqn{inc:final:1} and \eqn{inc:final:2} to equations
    \begin{equation}\label{eqn:increl1}
        \wedge^\ell(\cS_2/\cS_1)+\cS_1\star\wedge^{l-1}(\cS_2/\cS_1)=
        \begin{cases}
            \wedge^\ell(\cS_2) & \ell=1,\dots,n-2\\
            (1-q_1)\wedge^\ell(\cS_2) & \ell=n-1
        \end{cases}
    \end{equation}
    and 
    \begin{multline}\label{eqn:increl2}
        \wedge^\ell(\cS_2)+\wedge^{l-1}(\cS_2)\star(\C^n/\cS_2)\\=\wedge^\ell(\C^n)-\frac{q_2}{1-q_2}\wedge^{\ell-1}(\C^n/\cS_2)\star(\wedge^{\ell-1}(\cS_2)-\wedge^{\ell-1}(\cS_1))\quad\ell=1,\dots,n.
    \end{multline}
    Note that \eqn{increl1} and \eqn{increl2} are equivalent to the relations in \Theorem{rels}. Therefore, \(\widetilde{\Phi}\) induces the desired isomorphism \(\Phi:\overline{\QK}_T(X) \to \QK_T(X)\).
\end{proof}

\begin{remark}\label{remark:general}
    The construction of \(\overline{\QK}_T(X)\) can be generalized to any general partial flag variety \(X=\Fl(r_1,\dots,r_k;n)\). Following an analogous derivation in \cite{gu2023quantum}*{\S 4}, starting with the twisted superpotential 
    \begin{equation}
        \mathcal{W} = \sum_{j=1}^k \mathcal{W}_j,
    \end{equation}
    where
    \begin{align}
        \mathcal{W}_j ={}& \frac{r_j}{2} \sum_{a_j=1}^{r_j} \left(\ln X^{(j)}_{a_j}\right)^2 - \frac{1}{2}\left(\sum_{a_j=1}^{r_j} \ln X^{(j)}_{a_j}\right)^2 + \ln(-1)^{r_j - 1} q_j \sum_{a_j = 1}^{r_j} \ln X^{(j)}_{a_j}\cr
        &+ \sum_{a_j=1}^{r_j} \sum_{a_{j+1}=1}^{r_{j+1}} \mathrm{Li}_2\left(X^{(j)}_{a_j}/X^{(j+1)}_{a_{j+1}}\right),
    \end{align}
    and the (unsymmetrized) Coulomb branch equations
    \begin{equation}
        \exp\left(\frac{\partial \,\mathcal{W}}{\partial \ln X^{(j)}_{a_j}}\right) = 1, \quad j =1, 2, \dots, k,
    \end{equation}
    we have that in this case 
    \[
        \overline{\QK}_T(X)=\K_T(\pt)[\![q_1,\dots,q_k]\!][e_1(X^{(j)}),\dots,e_{r_j}(X^{(j)}),e_1(\overline{X}^{(j)}),\dots,e_{r_{j+1}-r_j}(\overline{X}^{(j)}),j=1,\dots,k]/J_q,
    \]
    where \(J_q\) is the ideal generated by
    \begin{equation}  \label{eq:keyreln}
        \sum_{r=0}^{r_{j+1}-r_j} e_{\ell-r}(X^{(j)}) 
        e_r(\overline{X}^{(j)})
        \: = \:
        e_{\ell}(X^{(j+1)}) \: + \:
        q_j \, e_{r_{j+1}-r_j}(\overline{X}^{(j)}) 
        e_{\ell - r_{j+1}+r_j}(X^{(j-1)})
        \end{equation} 
        for \(\ell=1,\dots,r_{j+1}\) and \(j=1,\dots,k\).
    
        Let \(S\) and \(I_q\) be as in \Definition{Iq}, then we can prove similarly that there is an isomorphism \(\overline{\QK}_T(X)\to S[\![q]\!]/I_q\) between the `Coulomb' and `Whitney' presentations given by: 
        \[
            e_\ell(X^{(j)})\mapsto e_\ell(X^{(j)}),\quad 1\leq \ell\leq r_j,\ 1\leq j\leq k,
        \]
        \[
            e_r(\overline{X}^{(j)})\mapsto 
            \begin{cases}
                e_r(Y^{(j)}) & 1\leq r< r_{j+1}-r_j\\
                e_{r_{j+1}-r_j}(Y^{(j)})/(1-q_j) & r=r_{j+1}-r_j
            \end{cases},\quad 1\leq j\leq k.
        \]
If we assume the \Cref{conj:lambda_y} above, then combining \Cref{thm:complete-intro} and the isomorphism 
above gives a geometric interpretation
of the variables from the `Coulomb' presentation.
\end{remark}

\section{The Whitney presentation for \texorpdfstring{$\QK_T(\Fl(n))$}{the equivariant quantum K ring of complete flag varieties}}\label{sec:fullflag}
In this section, we prove \Cref{conj2:lambda_y} in the case when $X=\Fl(n)=\Fl(1,\dots,n-1;n)$ is the complete flag variety, 
{\em under the assumption that \Conjecture{full} holds}. The main idea is to
rewrite the relations in \Cref{conj2:lambda_y} so that
only relations involving multiplication by $\wedge^i \cS_i$ appear (see \Lemma{reduction}),
then apply \Conjecture{full}. 
This reduction is not available for arbitrary partial flag varietys. 
Another key ingredient in our proof is to realize curve neighborhoods by 
certain iterated Demazure operators, a technique possibly of independent interest.

\subsection{A first reduction} In this subsection we rewrite the relations in \Cref{conj2:lambda_y} in a way that highlights the role of $\wedge^i \cS_i$. The results in this subsection are logically independent of \Conjecture{full}.

We start by recalling the relations \eqn{lambda_y_rel}:
\begin{equation*}\label{eqn2:lambda_y_full}
    \lambda_y(S_i)\star\lambda_y(\cS_{i+1}/\cS_i)=\lambda_y(\cS_{i+1})-y\frac{q_i}{1-q_i}\cS_{i+1}/\cS_i\star(\lambda_y(\cS_i)-\lambda_y(\cS_{i-1})).
\end{equation*}
After multiplying both sides by \(1-q_i\) and expanding, we can write the relations as 
\begin{equation}\label{eqn:quot_full}
        (1-q_i)(\wedge^\ell \cS_{i+1}-\wedge^\ell \cS_i)=\cS_{i+1}/\cS_i\star(\wedge^{\ell-1}S_i-q_i\wedge^{\ell-1}\cS_{i-1})
 \end{equation}      
 for $\ell=1,\dots, i+1; \ i=1,\dots,n-1$.

\begin{lemma}\label{lemma:reduction}
    Relations \eqn{quot_full} are equivalent to the relations
    \begin{equation}\label{eqn:det_full}
        \wedge^i S_i\star(\wedge^\ell \cS_{i+1}-\wedge^\ell \cS_i)=\wedge^{i+1}\cS_{i+1}\star(\wedge^{\ell-1}\cS_i-q_i\wedge^{\ell-1}\cS_{i-1}) \end{equation}
for $\ell=1,\dots, i+1; \ i=1,\dots,n-1$. 
\end{lemma}

\begin{proof}
    Note that 
    \begin{equation}\label{eqn:special_rel}
        \wedge^i \cS_i\star \cS_{i+1}/\cS_i=(1-q_i)\wedge^{i+1}\cS_{i+1}
    \end{equation} 
    is the \(\ell=1\) case of \eqn{det_full} and the \(\ell=i+1\) case of \eqn{quot_full}. It follows that 
    \begin{equation*}
        \det(\cS_i)\star\cS_{i+1}/\cS_i\star\dots\star\cS_{n-1}/\cS_{n-2}\star\C^n/\cS_{n-1}=(1-q_i)\cdots(1-q_{n-1})\wedge^n\C^n
    \end{equation*}
    and, in particular, \(\det(\cS_i)\) and \(\cS_{i+1}/\cS_i\) are units for \(i=1,\dots, n-1\).
    
    Multiplying both sides of \eqn{det_full} by \(\cS_{i+1}/\cS_i\) and using \eqn{special_rel}, we have
    \begin{multline*}
            (1-q_i)\wedge^{i+1}\cS_{i+1}\star(\wedge^{\ell}\cS_{i+1}-\wedge^{\ell}\cS_{i}) =
            \cS_{i+1}/\cS_i\star\det(\cS_i)\star(\wedge^\ell \cS_{i+1}-\wedge^\ell \cS_i)\\=\cS_{i+1}/\cS_i\star\wedge^{i+1}\cS_{i+1}\star(\wedge^{\ell-1}\cS_i-q_i\wedge^{\ell-1}\cS_{i-1})\quad\text{for }\ell=1,\dots, i+1,
    \end{multline*}
showing that \eqn{det_full} implies \eqn{quot_full}.
Multiplying both sides of \eqn{quot_full} by \(\det(\cS_i)\) and using \eqn{special_rel}, we have 
    \begin{multline}\label{eqn:rel_2}
            (1-q_i)(\wedge^\ell \cS_{i+1}-\wedge^\ell \cS_i)\star\det(\cS_i)=\det(\cS_i)\star\cS_{i+1}/\cS_i\star(\wedge^{\ell-1}\cS_i-q_i\wedge^{\ell-1}\cS_{i-1})\\
        =(1-q_i)\det(\cS_{i+1})\star(\wedge^{\ell-1}\cS_i-q_i\wedge^{\ell-1}\cS_{i-1})\quad\text{for }\ell=1,\dots, i+1,
    \end{multline}
    showing that \eqn{quot_full} implies \eqn{det_full}. 
\end{proof}

\subsection{Curve neighborhoods of tautological bundles} The main result in this subsection
is \Cref{cor:partialzd},
which 
{establishes an equality between 
curve neighborhoods of exterior powers of adjacent tautological bundles, when the degrees differ by a simple (co)root.}
This is the key calculation needed to
prove the QK relations in the next subsection. In fact, some of these results work 
for generalized flag varieties $G/B$, 
and we try to utilize suggestive 
notation so that it will be easy for the cognizant reader to rewrite 
the arguments in that generality.
This section is still logically independent of the validity of \Conjecture{full}.

Denote by 
$\alpha_i= (0, \ldots, 1, \ldots, 0)$ (with $1$ at position $i$) the degree of the (Schubert) curve
$X_{s_i}$. The degree $\alpha_i$ may also be identified with the simple (co)root 
$\varepsilon_i - \varepsilon_{i+1}$. More generally, denote by 
$R^+ = \{ \varepsilon_i - \varepsilon_j: 1 \le i < j \le n \}$ the set of positive roots of type \(A_{n-1}\), equipped with the  
partial order given by $\alpha \le \beta$ if $\beta - \alpha$ is a non-negative combination of simple
roots.

The Weyl group \(W\) is equipped with an associative monoid structure given by the {\em Demazure product}. 
For $u \in W$ and $s_i=(i,i+1)$ a simple reflection, 
\[ u \cdot s_i = \begin{cases} u s_i & us_i >u \/; \\
u & \textrm{ otherwise } \/. \end{cases}
\] 
More generally, if $v = s_{i_1} s_{i_2} \ldots s_{i_p}$ is a reduced expression, then
$u \cdot v = (((u \cdot s_{i_1}) \cdot s_{i_2}) \ldots \cdot s_{i_p})$. 
For $\beta \in R^+$, the support $\supp(\beta)$ is the set of simple roots $\alpha_i$ such that 
$\beta - \alpha_i \ge 0$. We will also implicitly utilize the fact that if $\alpha, \beta \in R^+$, then 
\begin{equation}\label{eqn:demcom} s_\alpha \cdot s_\beta = s_\beta \cdot s_\alpha \textrm{ if } \supp(\alpha) \subset \supp(\beta) \/, \textrm{ or } \supp(\beta) \subset \supp(\alpha) \/, \end{equation}
or $\supp( \alpha) \cap \supp(\beta) = \emptyset$ and the supports are not adjacent. It is easy to check this directly, and it also follows from \cite[Prop. 4.8]{buch.m:nbhds}.

Recall the definition of curve neighborhoods from \eqref{def:curvenbhd}. For an effective degree $d$,    
define the Weyl group element $z_d \in W$ by the requirement that the curve neighborhood $\Gamma_d(X_{s_0}) = X_{z_d}$, where \(s_0\) denotes the identity element in \(W\).
The element $z_d$ may be calculated recursively utilizing a formula
from \cite{buch.m:nbhds}: if $\beta$ is any maximal root such that $\beta\le d$, then 
\begin{equation}\label{eqn:reczd} z_d = z_{d - \beta} \cdot s_\beta \/. \end{equation}
The support $\supp(d)$ of $d=\sum d_i \alpha_i$ is defined similarly as the set of simple roots $\alpha_i$ such that $d_i \neq 0$. 
If $\alpha_i \in \supp(d)$, denote by $d' = d - \alpha_{i}$. 
In other words,
\[
    d_j'=
    \begin{cases}
        d_j & j\neq i\\d_i-1 & j=i.
    \end{cases}
\]
(The index $i$ will be understood from the context.)
The proof of the following lemma is a direct application of the recursive expression from \eqn{reczd}
and is left to the reader.

\begin{lemma}\label{lemma:d'} Let $d$ be a degree such that $d_i \neq 0$ and $d_{i+1}=0$. 
Assume that $z_d = s_{\beta_1}\cdot \ldots \cdot s_{\beta_k }\cdot \ldots \cdot s_{\beta_p}$ 
for some positive roots $\beta_1, \ldots, \beta_p$ such that 
$\alpha_i \in \supp(\beta_k) \subset \ldots \subset \supp(\beta_p)$ and this is the longest chain with this property. Then $\alpha_{i+1}$ is not in the support of any of the roots $\beta_j$, and, furthermore,
\[ z_{d'} = s_{\beta_1}\cdot \ldots \cdot s_{\beta_k - \alpha_i} \cdot \ldots \cdot s_{\beta_p} \/. \]
\end{lemma}
Let $p_i: \Fl(n) \to \Fl(1,\dots,\hat{i},\dots,n-1;n)$ be the natural projection. 
The {\em Demazure operator} is defined by $\partial_i = p_i^* {p_i}_*$, and it is a
$\K_T(\pt)$-linear endomorphism of $\K_T(\Fl(n))$. These operators satisfy 
$\partial_i^2=\partial_i$, and the
usual commutation and braid relations. In particular, if
$u= s_{i_1} \ldots s_{i_p}$ is a reduced decomposition of $u$, 
then there is a well-defined operator $\partial_u = \partial_{i_1} \ldots \partial_{i_p}$.
The Demazure operators satisfy $\partial_i \cO_u = \cO_{u \cdot s_i}$, therefore 
$\partial_{v^{-1}} \cO_u = \cO_{u \cdot v}$. We will utilize this to obtain curve neighborhoods 
by iterated Demazure operators. 

Recall that the degree $d$ curve neighborhood of a class $\sigma \in \K_T(\Fl(n))$ is
$\sigma[d]={\ev_2}_*(\ev_1^*(a))$. Using that $z_d = z_d^{-1}$ (cf.~\cite[Cor.4.9]{buch.m:nbhds}), 
we obtain that
\[ \cO_u[d]={\ev_2}_*(\ev_1^*(\cO_u))=\cO_{\Gamma_d(X_u)}=\cO_{u \cdot z_d} =  \cO_{u \cdot z_d^{-1}} = \partial_{z_d} (\cO_u) \/. \]
Since the Schubert classes form a basis, it follows that for any $\sigma \in \K_T(\Fl(n))$,
\begin{equation}\label{eqn:partial_zd} \sigma[d]={\ev_2}_*(\ev_1^*(\sigma)) = \partial_{z_d}(\sigma) \/. \end{equation} 
(This equation easily generalizes to any $G/B$.)
Our next goal is to calculate the curve neighborhoods of the bundles $\wedge^\ell \cS_i$ in \(\K_T(\Fl(n))\). 
To do this we need the following lemma. 
\begin{lemma}\label{lemma:dem-bundles} 
    For \(\ell=1,\dots,k\), the following hold:
    \begin{enumerate}
        \item\label{eqn:partial_wedge} \[ \partial_i (\wedge^\ell \cS_k) = \begin{cases} \wedge^\ell \cS_k & \textrm{ if } i \neq k \\\wedge^{\ell} \cS_{k-1} & \textrm{ if } i = k \/. \end{cases}\]
        \item\label{eqn:partial_product} For $i \neq k$ and any $\sigma \in \K_T(\Fl(n))$, \[ \partial_i(\wedge^\ell \cS_k \cdot \sigma) = \wedge^\ell \cS_k \cdot \partial_i(\sigma) \/. \]
    \end{enumerate}
\end{lemma}

\begin{proof} If $i \neq k$, then $\cS_k$ is also a bundle on $\Fl(1,\dots,\hat{k},\dots,n-1;n)$, therefore, by the projection formula
${p_i}_* (\wedge^\ell \cS_k) = \wedge^\ell \cS_k$. This implies part \eqn{partial_product} and the first branch of part \eqn{partial_wedge}.
If $i=k$, from the short exact sequence
$0 \to \cS_{k-1} \to \cS_k \to \cS_k/\cS_{k-1} \to 0$ it follows that 
\begin{equation}\label{eqn:wedge-eqs}\wedge^\ell \cS_k = \wedge^\ell \cS_{k-1} + \wedge^{\ell-1} \cS_{k-1}\cdot \cS_k/\cS_{k-1} \/. \end{equation} Then
\[ \begin{split} {p_k}_* ( \wedge^\ell \cS_k) & = {p_k}_*(\wedge^\ell \cS_{k-1} + \wedge^{\ell-1} \cS_{k-1}\cdot \cS_k/\cS_{k-1})\\
& = {p_k}_*(\wedge^\ell \cS_{k-1}) + \wedge^{\ell-1} \cS_{k-1}\cdot {p_k}_*(\cS_k/\cS_{k-1}) \\
& =  \wedge^\ell \cS_{k-1} \/. \end{split}
\]
Here the last equality follows because ${p_k}_*(\cS_k/\cS_{k-1})=0$ by \Proposition{relativeBWB}. The second branch of part \eqn{partial_wedge} follows from applying $p_k^*$ to this. 
\end{proof}
{
\begin{cor}\label{cor:partial_root}
    For \(\varepsilon_a-\varepsilon_{i+1}\in R^+\) and \(\ell=1,\dots,k\), the following hold:
    \[
        \partial_{\varepsilon_a-\varepsilon_{i+1}}(\wedge^\ell\cS_k)=\begin{cases}
            \wedge^\ell\cS_{a-1} & a\leq k\leq i\\\wedge^\ell\cS_k & \text{otherwise.}\\
        \end{cases}
    \]
\end{cor}
\begin{proof}
    The reflection \(s_{\varepsilon_a-\varepsilon_{i+1}}\) has reduced decomposition \(s_{i}s_{i-1}\ldots s_a\ldots s_{i-1}s_i\). The claim follows from repeated application of \Cref{lemma:dem-bundles}.
\end{proof}
}
\begin{cor}\label{cor:partialzd} Let $d$ be a degree such that $d_i \neq 0, d_{i+1}=0$. Then
\[ \partial_{z_d}(\wedge^{\ell} \cS_{i}) = \partial_{z_{d'}}(\wedge^{\ell} \cS_{i-1})\text{ for }\ell=1,\dots, i. \]
\end{cor}
\begin{proof} Let $d$ be an effective degree such that $d_i \neq 0$ and $d_{i+1}=0$. From
the recursive expression \eqn{reczd} we can write 
\[ z_d = s_{\beta_1}\cdot \ldots \cdot s_{\beta_k }\cdot \ldots \cdot s_{\beta_p} \]
 for some positive roots $\beta_1, \ldots, \beta_p$ such that 
 $\alpha_i\in \supp(\beta_k)\subseteq\dots\subseteq \supp(\beta_p)$,
 $\alpha_i\not\in \supp(\beta_j)$ for $1\leq j\leq k-1$, and 
 $\alpha_{i+1}\not\in \supp(\beta_j)$ for $1\leq j\leq p$.
In practice, this means that for any $k \le j \le p$, we have
$\beta_j = \varepsilon_{a_j} - \varepsilon_{i+1}$, where $i \ge a_k \ge a_{k+1} \ge \ldots \ge a_p$. By \Cref{lemma:d'}, we may write 
\[ z_{d'} = s_{\beta_1}\cdot \ldots \cdot s_{\beta_k - \alpha_i}\cdot \ldots \cdot s_{\beta_p} \/. \]
Define $z'\coloneqq s_{\beta_1}\cdot \ldots \cdot s_{\beta_{k-1}}$ and $z''=s_{\beta_k-\alpha_i}\cdot \ldots \cdot s_{\beta_p}$.
After writing 
$s_{\beta_k} = s_i \cdot s_{\beta_k - \alpha_i} \cdot s_i$, and observing that for any $j \ge k+1$, $s_i \cdot s_{\beta_{j}} = s_{\beta_{j}}\cdot s_i = s_{\beta_j}$ (by \eqn{demcom} above),
it follows that 
\[ z_d = z' \cdot s_i \cdot z'' \cdot s_i \textrm{ and } z_{d'} = z' \cdot z'' \/. \]
Since the indices $a_j > a_p -1$ for any $j \ge k$, it follows from {repeated application of \Corollary{partial_root}}
that
\[ \partial_{z'' \cdot s_i} (\wedge^{\ell} \cS_{i}) = \partial_{z''} (\wedge^\ell \cS_{i-1}) =
\wedge^{\ell} \cS_{a_p-1}\/, \] 
and since \(a_p-1<i\), we have
\[ 
\partial_{s_i \cdot z'' \cdot s_i} (\wedge^{\ell} \cS_{i}) = \partial_i \partial_{z''} \partial_i (\wedge^{\ell} \cS_{i}) =
\partial_i (\wedge^{\ell} \cS_{a_p-1}) = \wedge^{\ell} \cS_{a_p-1} \/. \]
In particular, this shows that $\partial_{s_i \cdot z'' \cdot s_i} (\wedge^{\ell} \cS_{i}) =
\partial_{z''} (\wedge^{\ell} \cS_{i-1})$. 
Then the claim follows from the calculation:
\[ \partial_{z_d} (\wedge^{\ell} \cS_{i}) = \partial_{z'} \partial_{s_i \cdot z'' \cdot s_i} (\wedge^{\ell} \cS_{i}) =
\partial_{z'} \partial_{z''} (\wedge^{\ell} \cS_{i-1}) = \partial_{z_{d'}} (\wedge^{\ell} \cS_{i-1}) \/. \]
\end{proof}
\subsection{The quantum K Whitney relations} In this subsection, we prove relations \eqn{det_full}
for the complete flag variety $X=\Fl(n)$ assuming  \Conjecture{full}. As a consequence, we get that \Conjecture{full} implies the Whitney presentation for \(\QK_T(X)\).

Let $d$ be an effective degree and \(\sigma, \tau\in\K_T(X)\). In light of \eqn{partial_zd}, \Conjecture{full} 
may be restated as 
\begin{equation}\label{eqn:KGW_full}
    \egw{\det(\cS_i),\sigma,\tau}{d}=\begin{cases}0 & d_i\neq0\\\chi_X^T\left(\det(\cS_i)\cdot\sigma\cdot\partial_{z_d}\tau\right) & d_i =0 \/\end{cases}
\end{equation}
for \(i=1,\dots,n-1\).
\begin{thm}\label{thm:full}For any $\sigma \in \K_T(X)$, \(\ell=1,\dots, i+1\), and \(i=1,\dots,n-1\), 
\Conjecture{full} implies  
    \begin{equation}\label{eqn:kgw_identity}
        \egw{\det(\cS_i), \wedge^\ell\cS_{i+1}-\wedge^\ell\cS_i,\sigma}{d}=\egw{\det(\cS_{i+1}),\wedge^{\ell-1}\cS_i,\sigma}{d}-\egw{\det(\cS_{i+1}),\wedge^{\ell-1}\cS_{i-1},\sigma}{d-\alpha_i} \/,
    \end{equation}
with the convention that
$\egw{\det(\cS_{i+1}),\wedge^{\ell-1}\cS_{i-1},\sigma}{d-\alpha_i} =0$ unless $\alpha_i \in   
\supp(d)$. 

In particular, the equalities \eqn{det_full} hold.
 \end{thm}
 \begin{proof}
We distinguish three cases. 
\begin{itemize} \item {\bf $d_i \cdot d_{i+1} \neq 0$}. In this case, both sides of \eqn{kgw_identity}
are equal to $0$ by \eqn{KGW_full}.
\item $d_i \neq 0,\ d_{i+1}=0$. By \eqn{KGW_full}, the left-hand side of \eqn{kgw_identity} is equal to $0$, and the 
right-hand side is equal to
\[ 
    \chi_X^T\left(\det(\cS_{i+1}) \cdot \left(\partial_{z_d} (\wedge^{\ell-1} \cS_i) - \partial_{z_{d-\alpha_i}} (\wedge^{\ell-1}\cS_{i-1})\right)\cdot\sigma\right). 
\]
Finally, note that by \Corollary{partialzd},
\(\partial_{z_d} (\wedge^{\ell-1} \cS_i) - \partial_{z_{d-\alpha_i}} (\wedge^{\ell-1}\cS_{i-1})=0.\)
\item \(d_i=0\). By \eqn{KGW_full} and \eqn{wedge-eqs}, 
the left-hand side of \eqn{kgw_identity} equals
\begin{align*}
    \chi_X^T\left(\det(\cS_i)\cdot(\wedge^\ell\cS_{i+1}-\wedge^\ell\cS_i)\cdot\partial_{z_d}\sigma\right)&=\chi_X^T\left(\det(\cS_i)\cdot(\cS_{i+1}/\cS_i)\cdot\wedge^{\ell-1}\cS_i\cdot\partial_{z_d}\sigma\right)\\
    &=\chi_X^T\left(\det(\cS_{i+1})\cdot\wedge^{\ell-1}\cS_i\cdot\partial_{z_d}\sigma\right),
\end{align*}
which is equal to the right-hand side of \eqn{kgw_identity}.\end{itemize}\end{proof}
{By \Lemma{reduction} and \Theorem{full} we have:
    \begin{cor}\label{cor:fln}
    \Conjecture{full} implies \Cref{conj2:lambda_y} for \(\Fl(n)\).
\end{cor}}

\appendix 
\section{Finite generation over formal power series}\label{app:fg}
The main result of this Appendix is \Cref{prop:KNak}, which is the key result 
needed in the proof of the presentation
of the quantum K ring in \Cref{thm:all_relations}. 
It gives mild conditions under which an 
{algebra over a formal power series ring}
is finitely generated as a module, allowing one to apply 
\Cref{prop:Nakiso}, or, more generally, Nakayama-type results. {One may also deduce \Cref{prop:KNak} directly from \cite[Exercise 7.4]{eisenbud:CAbook}.\begin{footnote} {We thank Prof. S. Naito for providing us with this reference.}\end{footnote} For the convenience of the reader, we include a proof.}

We start with the following general result
proved in \cite[Thm. 8.4]{matsumura:commutative}, see also \cite[\href{https://stacks.math.columbia.edu/tag/031D}{Tag 031D}]{StacksProj}.
\begin{lemma}\label{lemma:fgstacks} Let $A$ be a commutative ring and let $\mathfrak{a}\subset A$ be an ideal. 
Let $M$ be an $A$-module. Assume that $A$ is $\mathfrak{a}$-adically complete,
$\bigcap_{n \ge 1} \mathfrak{a}^n M = (0)$, and that $M/\mathfrak{a}M$ is a finitely generated $A/\mathfrak{a}$-module.
Then $M$ is a finitely generated $A$-module.\end{lemma}

For a commutative ring $S$ with $1$ 
we denote by $\Jac(S)$ its Jacobson radical,
i.e., the intersection of all its maximal ideals. It is proved in \cite[Prop. 1.9]{AM:intro}
that $x \in \Jac(S)$ if and only if $1-xy$ is a unit in $S$ for all $y \in S$. 

\begin{lemma}\label{lemma:fJac} 
    Let $R,S$ be commutative rings with $1$ and $\pi: R \to S$ be a surjective ring homomorphism with $\pi(1)=1$. Then:
    \begin{enumerate}
        \item $\pi(\Jac(R)) \subseteq \Jac(S)$;
        \item\label{item:ideal} If \(J\) is an ideal in \(R\), then $\pi(J)$ is an ideal in $S$.
    \end{enumerate}
\end{lemma}
\begin{proof} 
Let $x \in \Jac(R)$. Then \(1-xr\) is a unit in \(R\) for all \(r\in R\). This implies \(f(1-xr)=1-f(x)f(r)\) is a unit in \(S\). Since \(f\) is surjective, this means \(f(x)\in \Jac(S)\).

Part \eqref{item:ideal} is immediate from the definitions. 
\end{proof}

From now on, $S$ is a commutative Noetherian ring, and \(I\) is an ideal of the formal
power series ring $S[\![q_1, \ldots, q_k]\!]$. Let $$\pi: S[\![q_1, \ldots, q_k]\!] \to M\coloneqq S[\![q_1, \ldots, q_k]\!]/I$$ be the projection. Let 
\[ J\coloneqq \langle q_1, \ldots, q_k \rangle \subset S[\![q_1, \ldots, q_k]\!]. \]

\begin{lemma}\label{lemma:piJac} The ideal $\pi(J)$ is contained in the Jacobson radical of $M$.\end{lemma}
\begin{proof} By \cite[Prop. 10.15]{AM:intro} $J$ is contained in the Jacobson radical
of $S[\![q_1, \ldots, q_k]\!]$. Then the claim follows from \Cref{lemma:fJac}.\end{proof}
\begin{cor}\label{cor:separated-gen} We have that 
$\bigcap_{n \ge 1} \pi(J)^n = (0)$.\end{cor}
\begin{proof} Note that $S[\![q_1, \ldots, q_k]\!]$ is Noetherian from
\cite[Cor. 10.27]{AM:intro}. Then its quotient $M$ is also Noetherian, and by 
\Cref{lemma:piJac} 
we have that $\pi(J) \subset \Jac(M)$. The claim follows from a corollary a Krull's theorem,
\cite[Cor. 10.19]{AM:intro}, applied
to $M$ as a module over {$S[\![q_1, \ldots, q_k]\!]$}
and the ideal $\pi(J)$.
\end{proof}

Let us assume further that $S$ is an $R$-algebra for a Noetherian ring $R$. Let 
$$A\coloneqq R[\![q_1, \ldots, q_k]\!] \subset S[\![q_1, \ldots, q_k]\!]$$ with ideal $\mathfrak{a}=\langle q_1, \ldots, q_k\rangle \subset A$.

The goal of the Appendix is to prove the following Proposition, see also \cite[Exercise 7.8]{eisenbud:CAbook}.
\begin{prop}\label{prop:KNak} 
If $M/\mathfrak{a}M$ is a finitely generated \(A/\mathfrak{a}\)-module,
then 
\begin{enumerate}
    \item\label{item:fg} $M$ is a finitely generated $A$-module;
    \item\label{item:complete} \(M\) is \(\mathfrak{a}\)-adically complete.
\end{enumerate}
\end{prop}
\begin{proof} 
Note that $A$ is $\mathfrak{a}$-adically complete 
\cite[\S 7.1]{eisenbud:CAbook}, and that 
\[
    \bigcap_{n \ge 1} \mathfrak{a}^nM = \bigcap_{n \ge 1} J^n M= \bigcap_{n \ge 1} \pi(J)^n =(0)
\] 
by \Cref{cor:separated-gen}. 
Then part \eqref{item:fg} follows from \Cref{lemma:fgstacks}. Since $A$ is $\mathfrak{a}$-dically complete, it follows from \cite[Prop. 10.13]{AM:intro} that the $\mathfrak{a}$-adic completion of $M$ 
is $\widehat{M} = M \otimes_A \widehat{A}=M$, proving part \eqref{item:complete}. 
\end{proof}

\bibliographystyle{halpha}
\bibliography{biblio.bib}

\end{document}